\documentclass[a4paper]{article}
\usepackage[utf8]{inputenc}
\usepackage[T1]{fontenc}
\usepackage{amsmath, amssymb, amsthm}
\usepackage[dvipsnames]{xcolor}
\usepackage{graphicx}
\usepackage{csquotes}
\usepackage{tikz} 
\usetikzlibrary{angles,quotes,calc}
\usepackage{pgfplots}
\usepgfplotslibrary{colorbrewer}
\pgfplotsset{compat=1.18}
\usepackage{enumerate}
\usepackage{microtype}
\usepackage{mathdots}
\usepackage{hyperref}
\hypersetup{
  colorlinks,
  citecolor=Blue,
  linkcolor=Fuchsia,
  urlcolor=OliveGreen
}

\newtheorem{thm}{Theorem}[section]
\newtheorem{lem}[thm]{Lemma}
\newtheorem{defn}[thm]{Definition}
\newtheorem{prop}[thm]{Proposition}
\newtheorem{exam}[thm]{Example}
\newtheorem{quest}[thm]{Question}
\newtheorem{cor}[thm]{Corollary}
\newtheorem{rem}[thm]{Remark}

\RequirePackage{zref-clever}
\zcsetup{
    abbrev=true
    }

\newcommand{\Cref}[1]{\zcref[S]{#1}}
\zcRefTypeSetup{thm}{Name-sg = Theorem, name-sg = theorem, Name-pl =
  Theorems, name-pl = theorems}
\zcRefTypeSetup{lem}{ Name-sg = Lemma, name-sg = lemma, Name-pl =
  Lemmas, name-pl = lemmas}
\zcRefTypeSetup{prop}{ Name-sg = Proposition, name-sg = proposition,
  Name-pl = Propositions, name-pl = propositions}
\zcRefTypeSetup{cor}{ Name-sg = Corollary, name-sg = corollary,
  Name-pl = Corollaries, name-pl = corollaries}
\zcRefTypeSetup{defn}{ Name-sg = Definition, name-sg = definition,
  Name-pl = Definitions, name-pl = definitions}
\zcRefTypeSetup{exam}{ Name-sg = Example, name-sg = example, Name-pl =
  Examples, name-pl = examples}
\zcRefTypeSetup{quest}{ Name-sg = Question, name-sg = question,
  Name-pl = Questions, name-pl = questions}
\zcRefTypeSetup{rem}{ Name-sg = Remark, name-sg = remark, Name-pl =
  Remarks, name-pl = remarks}
\AddToHook{env/lem/begin}{\zcsetup{countertype={thm=lem}}}
\AddToHook{env/prop/begin}{\zcsetup{countertype={thm=prop}}}
\AddToHook{env/cor/begin}{\zcsetup{countertype={thm=cor}}}
\AddToHook{env/defn/begin}{\zcsetup{countertype={thm=defn}}}
\AddToHook{env/exam/begin}{\zcsetup{countertype={thm=exam}}}
\AddToHook{env/quest/begin}{\zcsetup{countertype={thm=quest}}}
\AddToHook{env/rem/begin}{\zcsetup{countertype={thm=rem}}}

\newcommand{\norm}[1]{\left\lVert#1\right\rVert}

\DeclareMathOperator{\arccot}{arccot}

\makeatletter
\newcommand{\itemref}[1]{%
  \textup{\hyperref[#1]{\tagform@{\ref*{#1}}}}%
}
\makeatother

\newcommand{\setsep}{\:|\:}

\newcommand{\abs}[1]{|#1|}
\newcommand{\bbR}{\mathbb{R}}
\newcommand{\bbC}{\mathbb{C}}

\newcommand{\bbN}{\mathbb{N}}
\newcommand{\calM}{\mathcal{M}}
\newcommand{\calS}{\mathcal{S}}

\newcommand{\calB}{\mathcal{B}}
\newcommand{\calQ}{\mathcal{Q}}

\newcommand{\Hone}{\mathbb{C}_{<1}}

\DeclareMathOperator{\real}{Re}
\DeclareMathOperator{\imag}{Im}
\DeclareMathOperator{\linSpan}{span}

\newcommand{\eps}{\varepsilon}
\newcommand{\iu}{\mathrm{i}\mkern1mu}
\newcommand{\mat}[1]{\begin{bmatrix}#1\end{bmatrix}}
\newcommand{\mars}{\includegraphics[width=.75em]{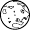}}

\usepackage[
backend=biber,
isbn=false,
giveninits=true,
style=alphabetic,
maxbibnames=99]
{biblatex}
\renewbibmacro{in:}{%
   \ifentrytype{article}{}{\printtext{\bibstring{in}\intitlepunct}}}
\AtEveryBibitem{%
   \clearlist{language}%
}
\newcommand{\xrep}{\xi}
\newcommand{\yimp}{\eta}
\newcommand{\ufunc}{u} 
\newcommand{\xvar}{r}
\newcommand{\yvar}{r}

\title{Moving-Average Iteration of Linear Maps}
\date{\today}
\author{Jeremias Epperlein\thanks{Faculty of Computer Science and Mathematics, University of Passau, Germany}, Fabian Wirth$^*$}

\begin{document}
\maketitle

\begin{abstract}                
We consider moving-average iteration for fixed complex matrices and place the theory in the general context of Krasnoselskii and Mann iteration. Necessary and sufficient spectral conditions for the asymptotic stability of moving-average iteration are presented in dependence on the fixed window width. For the corresponding region of stability in the complex plane explicit formulas are derived and it is shown that the region is open, convex and semi-algebraic. The regions are shown to be strictly increasing in the window length and 
the limiting stability region is determined as the window size tends to infinity.
\end{abstract}

\section{Introduction}
Iteration of functions in general is a cornerstone of computational mathematics. The theory of iterative processes has received significant attention with important concepts introduced in the 1950s, \cite{mannMeanValueMethods1953}, \cite{zbMATH03108780}. The analysis of iterative schemes has remained an active area ever since. In this paper, we address a stability question related to a certain moving-average process for linear maps that will be introduced in \Cref{sec:moving-average-iteration}. We begin by placing our approach into the general framework of iteration.
A continuous function $f: \bbR^d \to \bbR^d$ defines
two classical dynamical systems
which formalize the notion of
\enquote{at time $t$ move from the point $x$ in the direction of the point $f(x)$},
namely, the discrete time system
\begin{equation}
  x_{t+1} = f(x_t) \tag{D}
  \label{eq:discrete-time-system}
\end{equation}
and
the continuous-time system
\begin{align}
  \dot x(t) = f(x(t))-x(t). \tag{C}
  \label{eq:continuous-time-system}
\end{align}
Both systems have the fixed points of $f$ as their equilibria.
Discretisations of the continuous-time system of the form
\begin{align}
  x_{t+\delta} =x_{t} + \delta(f(x_t)-x_t) = \delta f(x_t) + (1-\delta)x_t \tag{P}
  \label{eq:discretisation}
\end{align}
interpolate between the dynamics of \eqref{eq:discrete-time-system} and
\eqref{eq:continuous-time-system}. The discrete-time systems
appearing in \eqref{eq:discretisation} are so-called
\emph{iteration schemes} applied to $f$, which are used
in various optimization and fixed-point-finding schemes to improve the convergence
speed. More precisely, \eqref{eq:discretisation} is an instance of the so-called \emph{Krasnoselskii
iteration scheme}\footnote{Note that we have switched the notation of the "time" index to $k$ now and as the remainder of the paper deals with variants of discrete-time systems, we will stick to this notation.}
\begin{align*}
  x_{k+1} = \omega_k f(x_k)+(1-\omega_k) x_k \tag{K}
  \label{eq:krasno}
\end{align*}
associated to a given sequence $(\omega_k)_{k \in \bbN_0}$ in $[0,1]$.

A further class of iteration schemes is given by the classical Mann iteration \cite{mannMeanValueMethods1953}
defined by
\begin{align}
  x_{k+1} = f\left(\frac{1}{k+1} \sum_{\ell=0}^{k}  x_\ell\right). \tag{M}
  \label{eq:Mann}
\end{align}
Here, instead of applying $f$ to $x_k$ directly,
first the arithmetic average $\overline{x}_k$ of all previous states is computed and then 
$f$ is applied to this average.
Note that the average itself satisfies
\begin{align}
  \overline{x}_{k+1} := \frac{1}{k+2} \sum_{\ell = 0}^{k+1} x_\ell = \frac{1}{k+2} f(\overline{x}_k) + \frac{k+1}{k+2}\
  \overline{x}_k. \tag{M'}
  \label{eq:Mann-Krasnoselskii}
\end{align}
which follows the dynamics of Krasnoselskii iteration
with weight sequence $\omega_k=\frac{1}{k+1}$.
Both \eqref{eq:Mann} and \eqref{eq:Mann-Krasnoselskii}
define non-autonomous dynamical systems to which
there are two natural autonomous approximations.
One of them is the Krasnoselskii iteration
with fixed weights $\omega_k=\frac{1}{n}$,
the other is what we call \emph{moving-average iteration}
\begin{align}
  x_{k+1} = f\left(\frac{1}{m} \sum_{\ell=0}^{m-1} x_{k-\ell}\right)
  \tag{A}\label{eq:moving-average-intro}
\end{align}
\emph{with window width $m \in \bbN$}.

In this paper we study 
iteration schemes for linear maps $f$
defined by a matrix $A \in \bbC^{d\times d}$.
Our original motivation for this came
from \cite{wirthNonhomogeneousPlacedependentMarkov2019}
where a probabilistic version of
moving-average iteration, motivated by the ergodic theorem, appeared in the context
of distributed optimization.

In particular we are interested in the
asymptotic stability of the origin
in the resulting linear systems.
It is well-known that the origin is
asymptotically stable for 
\eqref{eq:discrete-time-system} if
the spectrum $\sigma(A) \subseteq \bbC$
is contained in the open unit ball  $B_1(0)$ 
around $0$,
and that the origin is stable
under \eqref{eq:continuous-time-system}
if $\sigma(A)$ is contained
in $\Hone := \{ \lambda \in \bbC \setsep \real \lambda < 1\}$.
In both cases asymptotic stability of
the origin is characterized by the location
of the spectrum of $A$ in some open region of stability $\calS \subseteq \bbC$.
A simple calculation (see
\Cref{exa:stability-of-konstant-Krasnoselskii}\,(ii))
shows that this is also true
for Krasnoselskii iteration with constant
weight sequence  $\omega_k=\frac{1}{n}$.
Here the region of stability is
$\calS_n := B_n(1-n)$, the open ball of radius $n$ around
  $1-n \in \bbC$.
In particular, the sequence $\calS_n$ is increasing
in $n$ and the union $\bigcup_n \calS_n$
equals $\Hone$,
the region of stability for \eqref{eq:continuous-time-system}.

Under reasonable assumptions, an open region of stability also exists for Krasnoselskii iteration with
non-constant weights,
see \Cref{exa:stability-of-konstant-Krasnoselskii}. However,
we show in \Cref{exam:non-open-stability-region}
that, in general, stability conditions for Krasnoselskii iteration
with non-constant weights
cannot be characterized by spectral conditions.

Our main focus is on the region of stability
$\mars_m$ for moving-average iteration
with window width $m$.
These \emph{moving-average regions of stability} (MARS)
again form a sequence of bounded, convex, open subsets of $\Hone$,
which increase in $m$, but (to our surprise)
\begin{align*}
  \mars := \bigcup_{m \in \bbN} \mars_m
\end{align*}
is not equal to $\Hone$ but instead
equals
\begin{align*}
  \{\xrep+\iu\yimp{} \in \bbC\setsep \xrep{} < \yimp{} \cot \yimp{}, \ \yimp{}\in (-\pi,0) \cup(0,\pi)\}
  \cup (-\infty,1),
\end{align*}
see \Cref{fig:stability-regions}.
\begin{figure}
  \begin{center}
    \includegraphics[scale=0.5]{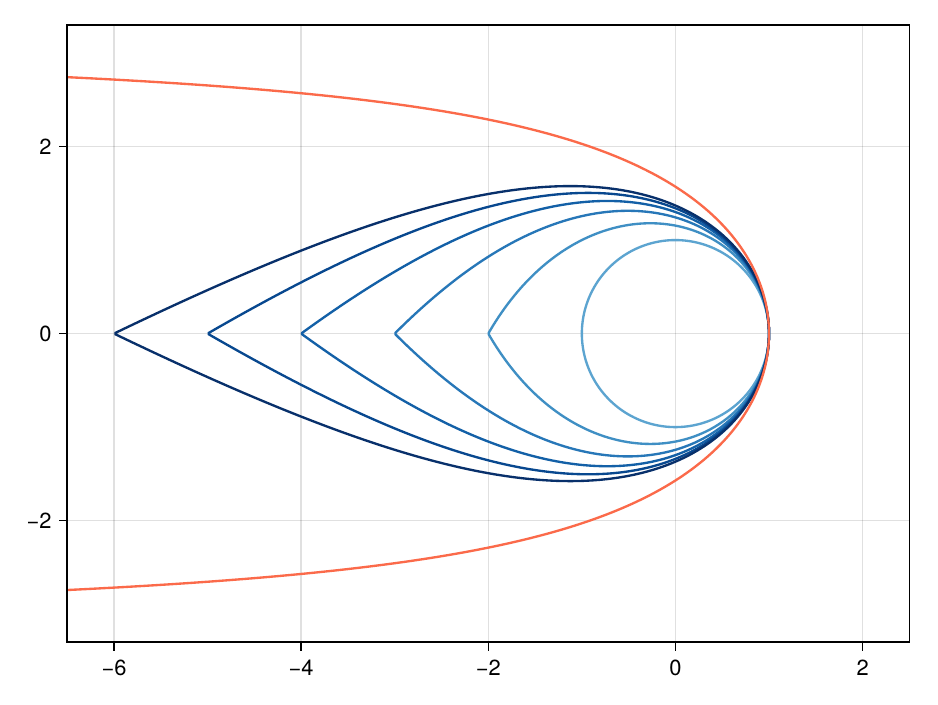}
    \vspace{-0.5cm}
  \end{center}
  \caption{Boundaries of the regions of stability $\mars_1,\dots,\mars_6$ in blue and $\mars$ in
    red}
  \label{fig:stability-regions}
\end{figure}

Iteration schemes appear under various names
throughout mathematics. Under the names used above
they primarily appear in the study of fixed-points
iterations. In particular, these iteration schemes have 
been studied extensively for nonexpansive maps
on convex subsets of Banach spaces.
A general account of this approach is given
in \cite{IterativeApproximationFixed2007} and \cite{bauschkeConvexAnalysisMonotone2017}.
In this line of research the focus is to find
conditions on the maps and the Banach space
to ensure (weak-) convergence
of the iteration schemes,
see for example \cite{reichWeakConvergenceTheorems1979a}.
A great deal is also known about convergence
speeds of these iteration schemes,
for example
using proof mining
\cite{kohlenbachUniformAsymptoticRegularity2003},
probabilistic techniques \cite{cominettiRateConvergenceKrasnoselskiiMann2014a}
or optimal transport
\cite{contrerasOptimalErrorBounds2023}.
Convergence results of this kind are used in such diverse fields 
as image reconstruction \cite{linKrasnoselskiiMannAlgorithmImproved2019a},
optimization \cite{cortildKrasnoselskiiMannIterations2025}
or
game theory \cite{belgioiosoConvergenceDiscreteTimeLinear2018a},
\cite{borgensRegularizedJacobitypeADMMmethods2019}.
A particularly important example 
of such an algorithm is 
the Douglas-Rachford splitting
which consists of the iteration of a convex combination 
of the identity and a composition of reflection operators,
see for example \cite{giselssonLinearConvergenceMetric2017}
or \cite{kanzowGeneralizedKrasnoselskiiMann2017}.
Similar ideas appear in various numerical schemes under the name \emph{under-relaxation} or \emph{damping}.
Up to now, the case of linear maps on the other hand, with the exception of
\cite{belgioiosoConvergenceDiscreteTimeLinear2018a}, has 
received surprisingly little attention in the literature
on iteration schemes.

For linear $f$, we can also
move the sum out of the function in \eqref{eq:moving-average-intro} and write the recursion as
\begin{align*}
    x_{k+m}-\sum_{\ell=0}^{m-1} \frac{1}{m} x_{k+\ell}
    = \sum_{\ell=0}^{m-1} \frac{1}{m} (f(x_{k+\ell})-x_{k+\ell}).
\end{align*}
This is a linear multistep method applied  
to the differential equation \eqref{eq:continuous-time-system}, see e.g. \cite[Section I.2]{iserlesFirstCourseNumerical2008}.
In this context what we are looking for
is called the region of \emph{(strict) asymptotic stability},
also known as the \emph{linear stability domain}
of this integration scheme, see
\cite[Section I.4.2]{iserlesFirstCourseNumerical2008}.
Recurrence equations like \eqref{eq:Mann}
and \eqref{eq:moving-average-intro}
and their stability
are also studied in the 
(delay) difference equations
literature, where they appear under the name \emph{Volterra difference equations}, see e.g.
\cite{elaydiAsymptoticStabilityLinear2007},
\emph{higher order difference equations}
\cite{berezanskySufficientConditionsGlobal2005}
or \emph{difference equations with several delays}
\cite{berezanskyStabilityLinearNonlinear2006}

The paper is structured as follows. We fix some notation in
\Cref{sec:notation}
and discuss our notion of stability for
 general linear recurrence relations in \Cref{sec:stability}.
This somewhat non-standard setting is needed to deal with arbitrary
Mann iteration schemes, where the whole past might influence the next
value.
In \Cref{sec:Krasnoselskii iteration} and \Cref{sec:mann-iteration} 
we discuss the regions of stability for general Krasnoselskii
and Mann iteration schemes.
We then introduce our central object of interest,
the moving-average regions of stability.
We determine some of their properties by elementary
means in \Cref{sec:moving-average-iteration}.
From the Schur-Cohn theorem, exact formulas for the regions are derived in
\Cref{sec:Schur-Cohn}.
Using these formulas, further properties are derived in
\Cref{sec:MARSstructure}.
In \Cref{sec:monotonicity} we show that the stability region becomes larger
when we increase the window width and we determine the exact limit
for the window width going to infinity.
The convexity of the moving-average regions of stability is 
established in \Cref{sec:convexity}.
In \Cref{sec:Banach-spaces} we finally take a quick glimpse into
the infinite-dimensional setting and show
that inclusion of the spectrum
in moving-average region of stability characterizes exponential stability
of the moving-average operator $\calM_n(A)$.
We finish with an outlook and some open problems in
\Cref{sec:conclusion}. Finally \Cref{sec:various-estimates}
contains some auxiliary estimates we used in  \Cref{sec:MARSstructure, sec:monotonicity, sec:convexity}.

\paragraph{AI disclosure:}{Qwen 3.8 27B was used for proofreading. Apart from that, the only use of LLMs
was a discussion with ChatGPT Free on September 1, 2026 about Theorem 2.10, which resulted in our proof strategy using polar coordinates.
All other arguments and calculations did not use LLM support.}

\section{Notation}
\label{sec:notation}
We denote the positive integers by $\bbN = \{1,2,\dots\}$
and the nonnegative integers by $\bbN_0$.
For a matrix $A \in \bbC^{d \times d}$ we
denote by $\sigma(A)$ its \emph{spectrum}, i.e.
the set of eigenvalues of $A$.
We denote the open ball of complex numbers with distance less than $r$ from $z
\in \bbC$ by $B_r(z)$. By $\norm{\cdot}$ we denote
the Euclidean norm on $\bbC^d$. The boundary of a set $M$ in a topological space is denoted $\partial M$.
For $c \in \bbR$ we define the half-plane $\bbC_{<c} = \{z \in \bbC \setsep \real{z}<c\}$. For a set $\calB \subseteq \bbC$
the set of complex conjugates is denoted $\calB^*:=\{\overline{b} \setsep b \in \calB\}$.
Similarly, for vectors $v \in \bbC^d$ and matrices $A \in \bbC^{d \times
  d}$ we set $v^* := \overline{v}^\top$ and $A^* :=
\overline{A}^\top$ where complex conjugation is applied entrywise.
For $w \in \bbC^m$ and $v$ a vector in a complex vector space $X$,
we define $w \otimes v$ as the vector in $X^m$
satisfying $(w \otimes v)_i = w_iv$.

\section{Stability of Linear Recurrence Equations}
\label{sec:stability}

We start by discussing the stability of general recurrence equations\footnote{
Such equations sometimes appear under the name
\enquote{Volterra difference equations of nonconvolution type}
in the literature, see e.g. \cite[Section 5]{elaydiStabilityAsymptoticityVolterra2009}.
} of the 
form
\begin{align}
    \label{eq:recurrence}
    x_{k+1} = \sum_{\ell=0}^{k} A(k,\ell)x_\ell,
\end{align}
where $A(k,\ell) \in \bbC^{d \times d}$, $k,\ell\in \bbN_0$,
are given matrices.

The solution of \eqref{eq:recurrence}
starting at time $k_0\in \bbN_0$ in $y=(y_0,\dots,y_{k_0})$
is the uniquely determined sequence 
$(x_k)_{k \in \bbN_0}$ in $\bbC^d$
satisfying \eqref{eq:recurrence}
for $k \geq k_0$ and $x_\ell = y_\ell$ for $\ell =0,\ldots, k_0$. Note that with this convention there exists a unique solution for all possible initial data $k_0\in \bbN_0, y \in \bbC^{d(k_0+1)}$.

\begin{defn}
\label{defn:general-nonautonomous}
  Let $\left(A(k,\ell)\right)_{k,\ell\in \bbN_0}\subseteq \bbC^{d\times d}$.
  We call a solution $(x_k)_{k \in \bbN_0}$
  of \eqref{eq:recurrence}
    \begin{enumerate}[(a)]
    \item \emph{stable}, if
    for every $\eps>0$ and $k_0 \in \bbN_0$
    there is $\delta>0$ 
    such that for every trajectory $(z_k)_{k \in \bbN_0}$
    with $\norm{z_{\ell}-x_{\ell}} < \delta$ for $\ell=0,\dots,k_0$
    we have
    \begin{align*}\norm{z_k - x_k}<\eps \text{ for all } k \in \bbN_0.
    \end{align*}
   
    \item
    \emph{globally attractive}, if
    every solution $(z_k)_{k \in \bbN_0}$ of \eqref{eq:recurrence}
    satisfies
    \[\lim_{k \to \infty} \norm{z_k -x_k} = 0.\]
  \end{enumerate}
\end{defn}

At first glance, the definitions of (asymptotic) stability resemble those that are commonly encountered in the study of dynamical systems. Note, however, that we have avoided to talk of the state of a dynamical system and there is a philosophical conundrum here, when considering standard axiomatizations of dynamical systems, see e.g. \cite[Chapter 2]{hinrichsenMathematicalSystemsTheory2010}. Following the dictum that the state is that information required to determine the evolution to the next time step, the state would have to be the entire history from zero up to time $k$ to determine $x_{k+1}$. However, it is then nontrivial to define a notion of state so that the state of the system converges for instance to the equilibrium position in $0$. For this reason, we are not talking about convergence of the state but rather about convergence of elements of sequences that are solutions. Interestingly, \cite[Definition 4.2]{zbMATH02171469} gives a definition of (asymptotic) stability exactly in the spirit of \cite[Chapter 3]{hinrichsenMathematicalSystemsTheory2010} and then goes on to discuss stability of Volterra equations in Section~6.3 without mentioning that this concept is in fact not defined in that reference. 
It is worth pointing out, that the problem disappears in the behavioral approach to systems theory introduced by Willems, \cite{willems1991paradigms}, \cite{polderman1998introduction}. In this approach a dynamical system $(T,W,\mathcal{B})$ is given by a time set $T$, the signal space $W$ and the \emph{behavior} $\mathcal{B} \subseteq W^T$. In the case of system \eqref{eq:recurrence}, we have $T=\bbN_0$, $W = \bbC^{d}$ and the behavior is given as the set of solutions of \eqref{eq:recurrence}. The point of behavioral theory is precisely to avoid an a priori introduction of a notion of state. (Asymptotic) stability can then be defined as in \Cref{defn:general-nonautonomous} and the problems of the standard dynamical systems approach are avoided\footnote{We thank Jochen Trumpf (Australian National University) for pointing out this connection.}.
Note that the stability condition (a) is not uniform in $k_0$.

\begin{prop}\label{prop:nonuniform-stability-notions}
  Let $\left(A(k,\ell)\right)_{k,\ell\in \bbN_0}\subseteq \bbC^{d\times d}$. The following are equivalent:
    \begin{enumerate}[(i)]
    \item There is a globally attractive solution $(x_k)_{k \in \bbN_0}$ of \eqref{eq:recurrence}.
    \item The zero solution is globally attractive for \eqref{eq:recurrence}.
    \item Every solution of \eqref{eq:recurrence} is stable and globally attractive.
    \end{enumerate}
\end{prop}
\begin{proof}
  Clearly $(iii) \implies (ii) \implies (i)$.
  To show $(i) \implies (iii)$, assume there is a globally attractive solution $(x_k)_{k \in \bbN_0}$.
  We first show that the zero solution is globally
  attractive.
  Let $k_0 \in \bbN_0$ and let $(z_k)_{k \in \bbN_0}$
  be any solution.
  By linearity $(x_k-z_k)_{k \in \bbN_0}$ is a solution
  and hence
  \[\lim_{k \to \infty} \norm{z_k}=\lim_{k \to \infty} \norm{x_k - (x_k-z_k)} = 0.\]
  This shows that the zero solution is globally attractive.
  Next we show that the zero solution is also stable.
  Fix $k_0 \in \bbN_0$ and let $e_1,\dots,e_n$ be a
  basis of $(\bbC^d)^{k_0+1}$. 
  There is a constant $D>0$
  such that the coordinates with respect
  to $e_1,\dots,e_n$ of every $y = (y_0,\ldots,y_{k_0})\in (\bbC^d)^{k_0+1}$
  with $\norm{y_k} \leq 1$ for $k =0,\dots,k_0$ are bounded in absolute value by $D$. 
  Let $w^j,\;j=1,\ldots,n$, be the solution with initial condition $e_j$
  at time $k_0$.
    
  Since $w^j$ converges to $0$ for
  all $j=1,\dots,n$, there is a uniform bound $C>0$ such that
  $\norm{w^j_k} \leq C$ for all $k\in \bbN_0, j \in
  \{1,\dots,n\}$.
  Now for every solution $(z_k)_{k \in \bbN_0}$ starting at $k_0$ with
  $\norm{z_k} \leq \delta := \frac{\eps}{nCD}$
  for $k=0,\dots,k_0$
  we can write $z$ as $\sum_{j=1}^n \beta_j w^j$
  with $\abs{\beta_j} \leq D\delta$
  and hence by linearity
  \begin{align*}
    \norm{z_k} \leq D\delta n C = \eps \text{ for all } k \in \bbN_0.
  \end{align*}
  This shows that the zero solution is stable. Global attractivity and stability of an arbitrary solution follow directly from linearity.
\end{proof}

It is not particularly surprising that attractivity implies stability
in the sense of Lyapunov, as this is true for many other classes of
linear systems, see e.g. \cite[Chapter
3]{hinrichsenMathematicalSystemsTheory2010}. Unfortunately, the proofs
we were able to locate in the literature did not encompass the case of
Volterra difference equations and the infinite delay present in this
case. In the sequel we will, by abuse of terminology, define
asymptotic stability through attractivity. For the linear systems
under consideration this
is justified by \Cref{prop:nonuniform-stability-notions}.

\section{Krasnoselskii Iteration}
\label{sec:Krasnoselskii iteration}
We now apply the stability notion defined in the previous \Cref{sec:stability} to
our first iteration scheme. Let $A \in \bbC^{d \times d}$ and let $(\omega_k)_{k \in \bbN_0}$ be a sequence in $[0,1]$.
The Krasnoselskii iteration of $A$ with
weight sequence $(\omega_k)_{k \in \bbN_0}$
is given by the recurrence equation
  \begin{align}
  \begin{aligned}
  \label{eq:krasnoselskii-eq}
    y_{k+1} &=\omega_k A y_k +  (1-\omega_k) y_k \\
    &=(\omega_kA + (1-\omega_k)I) y_k.
\end{aligned}
  \end{align}
This is a special case of the Volterra iteration \eqref{eq:recurrence} with matrices $A(k,k) = \omega_k A + (1-\omega_k)I$, and $A(k,\ell)=0$, $\ell=0,\ldots,k-1$, $k\in \bbN_0$. We continue to use the solution concept introduced in \Cref{sec:stability}.

\begin{defn}
\label{def:krasnostable}
    We call $A$ \emph{Krasnoselskii-stable}
    with respect to the weight sequence $\omega$
    (Krasnoselskii-$\omega$-stable), if  every solution of \eqref{eq:krasnoselskii-eq} converges to $0$.
\end{defn}

\begin{exam}
  \label{exa:stability-of-konstant-Krasnoselskii}
  \begin{enumerate}[(i)]
  \item Krasnoselskii iteration of $A$ with constant weight $\omega_k=1$
    is just conventional iteration\footnote{Known
    as \enquote{Picard iteration} \cite{IterativeApproximationFixed2007} or even \enquote{Banach-Picard iteration} \cite{bauschkeConvexAnalysisMonotone2017} in the literature on iteration schemes.} of $A$, 
    \[y_{k+1}=Ay_k.\]
    Therefore $A$ is $Krasnoselskii$ stable
    with constant weight $1$
    if and only if $A$ is Schur stable, i.e. if and only if \[\sigma(A) \subseteq B_1(0).\]
  \item If the sequence $(\omega_k)_{k \in \bbN_0}$ is constant equal
    to $\omega \in (0,1]$ then Krasnoselskii iteration
    of $A$ with weight $\omega$ corresponds to
    the conventional iteration of $(1-\omega)I + \omega A$
    which is Schur stable if and only if
    $(1-\omega) + \omega \sigma(A) \subseteq B_1(0)$,
    i.e. if and only if \[\sigma(A) \subseteq B_{\frac{1}{\omega}}\left(1-\frac{1}{\omega}\right).\]
  \item Note that it really matters that we
    ask for convergence for all initial times.
    For example consider the sequence $\omega_0 = \frac{1}{3}, \omega_n=1$
    for $n \geq 1$.
    Then Krasnoselskii iteration of $A=-2$ with weight $(\omega_n)_{n \in
      \bbN_0}$ is not asymptotically stable according to the
      definition we use.
    Nevertheless, the Krasnoselskii iteration of $A=-2$
    with weight $(\omega_n)_{n \in
      \bbN_0}$ initialized at time
    $k_0=0$ with arbitrary initial value $y_0\in \bbC$ is constant equal
    to $0$ after the first time step since
    \[y_1 = \left(\left(1-\frac{1}{3}\right) +\frac{1}{3}(-2)\right)y_0=0.\]
  \end{enumerate}
\end{exam}

Next we discuss the Krasnoselskii stability of matrices
for weight sequences converging to zero. Part (i)
appeared as Theorem 2 in \cite{belgioiosoConvergenceDiscreteTimeLinear2018a}.
The proof given in this reference uses results about nonexpansive operators.
We give a short self-contained proof. We first note an auxiliary lemma.
\begin{lem}
    \label{lem:bomega-lemma}
    Let $\omega = (\omega_k)_{k \in \bbN_0}$ be a sequence in $[0,1]$
  converging to $0$ and
    let $b\in \bbC_{<0}$. Then there exists a $K_1\in \bbN_0$ such that for all $K_1\leq k_0 \leq K \in \bbN_0$ we have
    \begin{equation}
        \label{eq:bomega-ineq}
        \prod_{k=k_0}^K\abs{1+b\omega_k}^2 \leq
        \exp\left(\sum_{k=k_0}^K \omega_k \real b \right).
    \end{equation}
\end{lem}
\begin{proof}
For $\omega \in \bbR$ we have 
    \begin{equation}
    \label{eq:bsquaredformula}
      \abs{1+b\omega}^2 = (1+b\omega)(1+\overline{b}\omega) = 1 + 
      2\omega \real b + \omega^2|b|^2.
  \end{equation}
  Using the estimate $(1+x) \leq e^x$, valid
  for $x\in \bbR$, we get for arbitrary $k_0 \leq K$
  \begin{equation}
  \label{eq:firstexpestimate}
     \prod_{k=k_0}^K\abs{1+b\omega_k}^2 \leq
        \exp\left(\sum_{k=k_0}^K 2\omega_k \real b+\omega_k^2 \abs{b}^2\right).
  \end{equation}
  Choose $K_1\in \bbN_0$ sufficiently large such that $\omega_k \leq - \real b \abs{b}^{-2}$ for all $k\geq K_1$. Then the claim \eqref{eq:bomega-ineq} holds for all $K \geq k_0 \geq K_1$.
\end{proof}

\begin{thm}
  \label{thm:stability-Krasnoselskii}
  Let $\omega = (\omega_k)_{k \in \bbN_0}$ be a sequence in $[0,1]$
  converging to $0$.
  \begin{enumerate}[(i)]
    \item If $\sum_{n=0}^\infty \omega_n = \infty$,
  then $A$ is 
  Krasnoselskii-stable
  with weight sequence $\omega$
  if and only if $\sigma(A) \subseteq \Hone$.
\item If $\sum_{n=0}^\infty \omega_n < \infty$,
  then no matrix is Krasnoselskii-stable with weight sequence $\omega$.
\end{enumerate}
\end{thm}
\begin{proof}
  We first treat (i) in the case of $A=a \in \bbC$.
 To simplify notation, set $b:=a-1$.
  By \Cref{def:krasnostable},  Krasnoselskii-$\omega$-stability of $a$ 
  is equivalent to  
  \begin{align*}
   \forall k_0 \in \bbN_0 \;:\;\lim_{K\to \infty} \prod_{k=k_0}^K \abs{(1-\omega_k)+\omega_k a} =0.
  \end{align*}
  Squaring the partial products
  does not change their (non-) convergence.
  In view of \eqref{eq:bsquaredformula}, we want to show that $\sum_{n=0}^\infty \omega_n =
  \infty$ 
  implies the validity of the equivalence
  \[\left(\forall k_0\in \bbN_0: 
    \prod_{k=k_0}^\infty (1+2\omega_k \real
    b+\omega_k^2 \abs{b}^2)=0\right) \quad \Leftrightarrow \quad b \in \bbC_{<0}.\]
  A necessary condition for the convergence of the product on the left to zero
  is that
  \begin{equation}
  \label{eq:bs-ness}
      \omega_k (2\real b+\omega_k \abs{b}^2)<0
  \end{equation} for infinitely many $k$. As $\omega_k\to 0$, 
  this is equivalent to $\real b <0$, i.e., $b \in \bbC_{<0}$.
  For the converse direction, \Cref{lem:bomega-lemma} shows that if $\real b <0$ and $\sum_{k=0}^\infty \omega_k =
  \infty$, then 
  the sum in the exponent on the right hand side of \eqref{eq:bomega-ineq} diverges to $-\infty$,
  which means that our product converges to $0$.
  
  Now consider the general case $A \in \bbC^{d \times d}$.
  By passing to the restriction to eigenspaces we immediately see
  from 
  the arguments around \eqref{eq:bs-ness}
  that asymptotic stability of the 
  $(\omega_n)_{n \in \bbN}$-Krasnoselskii iteration implies
  $\sigma(A) \subset \Hone$.
  For the converse implication note that using
  the Jordan normal form we can reduce the 
  discussion
  to the case
  $A = aI + N$
  where $N$ is the $d \times d$ Jordan block with eigenvalue $0$.
  Once again setting $b:=a-1$, we get
  \[(1-\omega_k)I+\omega_k A = (1+b\omega_k)I + \omega_k N , \quad  k \in \bbN_0.\]
  Estimating the norm we obtain for 
  $K-k_0> d$, $k_0$ sufficiently large so that $b\omega_k \neq -1$ and $k\geq k_0$,
  \begin{align*}
    \norm{\prod_{k=k_0}^K (I-\omega_k)I+\omega_k A} \leq
    \left(\prod_{k=k_0}^K\abs{1+b\omega_k}\right) \sum_{\ell=0}^{d-1}
    \norm{N^\ell} e_\ell\left(\frac{\omega_{k_0}}{1+b\omega_{k_0}},\dots, \frac{\omega_K}{1+b\omega_K}\right),
  \end{align*}
    where $e_\ell(x_1,\dots,x_{K-k_0+1}) = \sum_{1\leq n_1 < \dots <n_\ell \leq
    K-k_0+1} x_{n_1}\cdots x_{n_\ell}$ is the $\ell$-th
    elementary symmetric polynomial
    in $K-k_0+1$ variables.
  Now $\frac{1}{1+b\omega_k}$ converges
  to $1$, so for $k$ sufficiently large it is bounded on $[k,\infty)$ by some constant $C>0$.
  Putting this into our estimate we obtain for $k_0$ sufficiently large
    \begin{align*}
    \norm{\prod_{k=k_0}^K (I-\omega_k)I+\omega_k A}
    &\leq
    \left(\prod_{k=k_0}^K\abs{1+b\omega_k}\right) \sum_{\ell=0}^{d-1} \norm{N^\ell}
    \left(C\sum_{k=k_0}^K \omega_k\right)^\ell \\
    &\stackrel{\eqref{eq:bomega-ineq}}{\leq}
      \exp\left(\frac{1}{2}\real b  \sum_{k=k_0}^K \omega_k \right)\sum_{\ell=0}^{d-1} \norm{N^\ell}
      \left( C \sum_{k=k_0}^K \omega_k\right)^\ell.    
  \end{align*}
  The right hand side converges to $0$ for $K \to \infty$ since $\lim_{r \to \infty} \exp(-\beta r)r^\ell =
  0$ for all $\beta>0$ and $\ell \in \bbN$.
  
  For (ii) the restriction to eigenspaces shows again that
  it is enough to treat the scalar case $A=:b+1 \in \bbC$.
  Assume $\sum_{k=0}^\infty \omega_k<\infty$  and that $b+1$ is 
  Krasnoselskii-$\omega$-stable.
  Appealing again to \eqref{eq:bsquaredformula} this means
  \[\prod_{k=k_0}^\infty 1+2\real b \omega_k + \omega_k^2 \abs{b}^2 =0\]
  for all $k_0 \in \bbN_0$.
    As in \eqref{eq:bs-ness}, this implies
  $\real b<0$. Set \[c_k:=-(2\real b \omega_k + \omega_k^2
  \abs{b}^2).\]
  For sufficiently large $k$ we have $c_k \in (0,1)$.
  Hence
  \begin{align*}
    \exp\left(\sum_{k=k_0}^\infty \frac{c_k}{1-c_k}\right) \geq   \prod^\infty_{k=k_0} 1+\frac{c_k}{1-c_k}=\prod_{k=k_0}^\infty \frac{1}{1-c_k} = \infty.
  \end{align*}
  for sufficiently large $k_0$.
  Since $\lim_{k \to \infty} 1-c_k = 1$, this implies
  $\sum_{k=k_0}^\infty c_k = \infty$.
  But $\sum_{k=k_0}^\infty \omega_k^2 \leq \sum_{k=k_0}^\infty \omega_k <\infty$,
  and so also $\sum_{k=k_0}^\infty c_k<\infty$.
  This contradiction shows the assertion.
\end{proof}

For weight sequences $(\omega_k)_{k \in \bbN_0}$ not converging to $0$,
Krasnoselskii stability of $A$
is not characterized by the location of the
spectrum of $A$, as the following example shows.
\begin{exam}
  \label{exam:non-open-stability-region}
  Consider the sequence $\omega$ given by $\omega_k = \frac{k}{k+1}$, $k\in \bbN_0$.
  Then $-1$ is 
    Krasnoselskii-$\omega$-stable, but
  $\left(\begin{smallmatrix}
    -1 & 1 & 0\\
    0 & -1 &1\\
    0& 0 &-1 
  \end{smallmatrix}\right)$
  is not. Furthermore,
  the set of all $a \in \bbC$, which
  are not Krasnoselskii-$\omega$-stable   is not open.
\end{exam}
\begin{proof}
  We first show
  that $a \in \bbC$ is Krasnoselskii-stable 
  with respect to $(\omega_k)_{k \in \bbN_0}$  if
  $\abs{a} \leq 1$ and $a \neq 1$.
  Set $b:=a-1$.
  The necessary condition (already used in \eqref{eq:bs-ness})
  \[\omega_k (2\real b+\omega_k \abs{b}^2)<0 \text{ for infinitely many } k\]
  implies $b=a-1 \neq 0$ and $\abs{a}-1=\abs{b+1}^2-1 = 
  2\real b+ \abs{b}^2 \leq 0$.

  On the other hand, for $\abs{a} \leq 1$ and $b \neq 0$
  we have $2\real b+ \omega_k \abs{b}^2 < 0$
  and hence, using \eqref{eq:firstexpestimate}, we have for arbitrary $k_0 \leq K$
  \begin{align*}
    \prod_{k=k_0}^{K-1}  \abs{1+b \omega_k}^2 \leq \exp\left(\sum_{k=k_0}^{K-1} 2\left(1-\frac{1}{k+1}\right)\real b+ \left(1-\frac{1}{k+1}\right)^2 \abs{b}^2\right).
  \end{align*}
  Using the following well-known bounds for the 
  partial sums of the harmonic series
  \begin{align*}
    \log(K) \leq \sum_{n=1}^{K} \frac{1}{n} \leq \log(K)+1
  \end{align*}
  together with $\sum_{n=1}^\infty \frac{1}{n^2} < \infty$, it follows that there is a constant $C$, depending on $b$
  and $k_0$  such that 
  \begin{align*}
    \prod_{k=k_0}^{K-1}  \abs{1+b \omega_k}^2 &\leq \exp((K- \log K )2\real b + (K -
    2\log K )\abs{b}^2 +C) \\
    &\leq \exp((K-\log K) (2\real b + \abs{b}^2) - \log K \abs{b}^2 +C) \to 0
  \end{align*}
  for $K \to \infty$.
  In particular for $a=-1$ we get asymptotic stability for the
  Krasnoselskii iteration.
  On the other hand, let $N$ be the $3 \times 3$ Jordan block with
  eigenvalue $0$ and set $A:=-I+N$.
  Then for $K\geq k_0 \geq 2$ and using $\beta_k:=\frac{\omega_k}{1-2\omega_k}=\frac{k}{1-k}$
  \begin{align}
  \label{eq:2-2-counter-krasnoselskii}
  \begin{aligned}
    &\prod_{k=k_0}^K ((1-\omega_k)I+\omega_k A)=\prod_{k=k_0}^K ((1-2\omega_k) I
                                     + \omega_k N)\\
    &=\left(\prod_{k=k_0}^K \frac{1-k}{k+1}\right) \left(I + \sum_{k=k_0}^K
      \beta_{k} N + \sum_{k_0\leq k_1<k_2\leq K}
      \beta_{k_1} \beta_{k_2} N^2\right).
  \end{aligned}
  \end{align}
  Now 
  \begin{align*}
      \left\vert\prod_{k=k_0}^K \frac{1-k}{k+1}\right\vert \geq \frac{1}{K(K+1)}
  \end{align*}
  and 
  \begin{align*}
      \sum_{k_0\leq k_1<k_2\leq K} \beta_{k_1} \beta_{k_2} \geq \frac{1}{4}\frac{(K-k_0)^2}{2}
  \end{align*}
  since the left hand side is a sum of at least $\frac{(K-k_0)^2}{2}$ summands, each of which is at least $\frac{1}{4}$.
  Hence the absolute value of the coefficient of $N^2$ in
  \eqref{eq:2-2-counter-krasnoselskii}
  is at least
  \begin{align*}
    \frac{1}{K(K+1)} \frac{(K-k_0)^2}{8}
  \end{align*}
  which converges to $\frac{1}{8} \neq 0$ for $K \to \infty$.
  
\end{proof}

\section{Mann Iteration}
\label{sec:mann-iteration}
In \cite{mannMeanValueMethods1953}, Mann  described a more general iteration scheme than the classical Mann iteration \eqref{eq:Mann}
based on what we call a \emph{Mann-weight sequence}.
This is a sequence $(\pi^k)_{k \in \bbN_0}$ of coefficients of convex combinations, i.e.  a sequence with
  values in $[0,1]$
  such that $\pi^k_\ell =0$ for $\ell>k$
  and $\sum_{\ell=0}^k \pi^k_\ell=1$ for every $k \in \bbN_0$.
  Now given a matrix $A \in \bbC^{d \times d}$
  the \emph{Mann iteration} of $A$ with
    weight sequence $(\pi^k)_{k \in \bbN_0}$
    is given by the recurrence equation
  \begin{align}
  \label{eq:mann-eq}
    y_{k+1} &= A\sum_{\ell=0}^k \pi^k_\ell y_\ell .
  \end{align}
  The notation that we use is inspired by
  \cite{contrerasOptimalErrorBounds2023}. Again this is a special case of Volterra iteration and we continue to use the solution and stability concepts introduced in \Cref{sec:stability}.

  \begin{defn}
    We call $A\in \bbC^{d\times d}$ \emph{Mann-stable}
    with respect to the weight sequence $(\pi^k)_{k \in \bbN_0}$
    if every solution of \eqref{eq:mann-eq} converges to $0$.
\end{defn}

\begin{rem}
\label{rem:mann-only-eventual-weights-are-relevant}
  It follows immediately from the definition, that a matrix 
  that is Mann-stable with respect to a weight sequence
 $(\pi^k)_{k \in \bbN_0}$ is also Mann-stable
  with respect to any weight sequence which
  eventually agrees with $(\pi^k)_{k \in \bbN_0}$.
\end{rem}

\begin{exam}
  The classical Mann iteration \eqref{eq:Mann} given by
  \begin{align*}
    y_{k+1} = A \left(\frac{1}{k+1} \sum_{\ell=0}^k y_\ell\right)
  \end{align*}
  discussed in the introduction
  corresponds to the weight sequence $(\pi^k)_{k \in \bbN_0}$ with $\pi^k = (\frac{1}{k+1},\dots,\frac{1}{k+1},0,\dots)$.
\end{exam}

Just as we did in the introduction, we set 
\begin{align}
  \overline{y}_k:=\sum_{\ell=0}^k \pi^k_\ell y_\ell,
  \label{eq:def-overlinex}
\end{align}
and obtain
\begin{align}
  y_{k+1} &= A \overline{y}_k
            \label{eq:yk-eq-A-overlineyk}
\end{align}
If there is $\overline{y}_{-1}$ with $A \overline{y}_{-1}=y_0$, then we get the following recursion
for $\overline{y}_k$:
\begin{align*}
  \overline{y}_k = \sum_{\ell=0}^k \pi_\ell^k A\overline{y}_{\ell-1}
\end{align*}
Notice that we did not use the linearity of $A$ here.
This is the form of Mann iteration used in \cite{contrerasOptimalErrorBounds2023}.

Krasnoselskii iteration as discussed in \Cref{sec:Krasnoselskii
  iteration} can be treated as a special case of Mann iteration as the
following proposition shows. This connection
already appears in Mann's original paper
\cite{mannMeanValueMethods1953}
for the classical Mann iteration.
For a discussion of the general case
see \cite[Theorem 4.2]{IterativeApproximationFixed2007}. 
Note, however, that we were not able to find the exact statement about the connection between the stability of the two iteration schemes in the literature.      

\begin{prop}
  \label{prop:Krasnoselskii-as-mann}
  Let $(\omega_k)_{k \in \bbN_0}$ be a sequence in $[0,1]$.
  Define $\pi^k$ recursively by
  \begin{align*}
    \pi^0 &:= \delta^0, \quad \pi^{k+1} :=(1-\omega_k) \pi^k +  \omega_k \delta^{k+1}
  \end{align*}
  where $\delta^k$ is the sequence which equals one at index $k$
  and zero otherwise.
  
  \begin{enumerate}[(i)]
  \item Let $(x_k)_{k \in \bbN_0}$
    be the solution of the Mann iteration of $A$ 
    with respect to the weight sequence 
    $(\pi^k)_{k \in \bbN_0}$ with initial values
    $x_0,\dots,x_{k_0}$ started at time $k_0$.
    Then
    \begin{align}
      \overline{x}_{k+1} = (1-\omega_k) \overline{x}_k+\omega_k
      A\overline{x}_k
      \label{eq:overlinexk-krassno}
    \end{align}
      for all $k \geq k_0$.
  \item Assume that $A$ is invertible or that $\sum_{k=0}^\infty
    \omega_k = \infty$. Then $A$ is Krasnoselskii-stable with respect to $(\omega_k)_{k \in \bbN_0}$
  if and only if $A$ is Mann-stable with respect to $(\pi^k)_{k \in \bbN_0}$.
  \end{enumerate}
\end{prop}

\begin{rem}
    Some references simply call
    Mann iteration with the weights
    defined in \Cref{prop:Krasnoselskii-as-mann}
    \enquote{Krasnoselskii iteration}, see e.g.
    \cite{contrerasOptimalErrorBounds2023}, \cite{cortildKrasnoselskiiMannIterations2025} and this 
    is justified by \eqref{eq:overlinexk-krassno}. However,
    without one of the assumptions in (ii) 
    the stability of the two
    iteration schemes differs. This is due to the fact
    that $\overline{x}_k \to 0$
    does not imply $x_k \to 0$ in general.
    For example, if $A=0$, then Mann iteration is always stable regardless of the weights. Krasnoselskii iteration,
    however, is only stable if the weights go to
    zero sufficiently fast as we have seen in \Cref{thm:stability-Krasnoselskii}.
\end{rem}
 \begin{proof}[Proof of {\Cref{prop:Krasnoselskii-as-mann}}]
   \begin{enumerate}[(i)]
   \item
     This is just a simple calculation:
     \begin{align*}
       \overline{x}_{k+1}
       &=\sum_{\ell=0}^{k+1} \pi^{k+1}_\ell x_\ell
         =\sum_{\ell=0}^{k} (1-\omega_k)\pi_\ell^{k} x_\ell + \omega_k x_{k+1}  \\
       &= (1-\omega_k)\overline{x}_k + \omega_k A \overline{x}_k .
        \end{align*}
 \item Let $(x_k)_{k \in \bbN_0}$
    be the solution of the Mann iteration of $A$ 
    with weight sequence 
    $(\pi^k)_{k \in \bbN_0}$ and initial values
    $x_0,\dots,x_{k_0}$ started at time $k_0$.
 We first show that the assumptions imply that $\overline{x}_k \to 0$ if and only if $x_k \to 0$.
   The forward implication follows directly from
   \eqref{eq:yk-eq-A-overlineyk}.
   If $A$ is invertible, then the reverse implication follows
   in the same way.
   For the remaining case assume $x_k \to 0$
   and $\sum_{k=0}^\infty \omega_k = \infty$.
   This directly implies $0\leq \prod_{k=0}^\infty (1-\omega_k) \leq
   \exp(-\sum_{k=0}^\infty \omega_k)=0$.
   By induction we see that for $\ell \leq k$ we have
   \begin{align*}
     \pi^{k}_\ell &= \omega_{\ell-1} \prod_{j=\ell}^{k-1}
                    (1-\omega_j)
   \end{align*}
   where we set $\omega_{-1}:=1$. Hence for every $\ell \in \bbN_0$ we have
   $\lim_{k \to \infty} \pi^k_\ell= 0$.
   Thus for $\eps>0$,
   there is $k_1 \in \bbN_0$ such that
   $\norm{x_k}\leq \frac{\eps}{2}$ for all $k \geq k_1$
   and there is $k_2 \geq k_1$ such that
   $\pi^{k}_\ell \leq \frac{\eps}{2k_1\max\{1,\norm{x_0},\dots,\norm{x_{k_1-1}}\}}$
   for all $k \geq k_2$ and $\ell \in \{0,\dots,k_1-1\}$.
   Then for $k \geq k_2$
   \begin{align*}
     \norm{\overline{x}_k} &= \norm{\sum_{\ell=0}^k \pi^k_\ell x_\ell} \\
     &\leq \frac{\eps}{2}+\frac{\eps}{2}.
   \end{align*}
   Thus $\overline{x}_k \to 0$.

   If $A$ is Krasnoselskii-stable (for brevity we omit the weights
   in the remainder of the proof)
   and $x_k$ is a solution of the Mann iteration of $A$,
   then $\overline{x}_k \to 0$ by \eqref{eq:overlinexk-krassno}
   and hence also $x_k \to 0$ as shown above.

   On the other hand, assume that $A$ is Mann-stable
   and $y_k$ is a solution of
   the Krasnoselskii iteration of $A$
   started at $k_0$.
   There is $\ell \in \{0,\dots,k_0\}$
   such that $\pi^{k_0}_{\ell} \neq 0$.
   Let $x$ be the solution of
   the Mann iteration of $A$ with
      \[x_k = \begin{cases}
     (\pi^{k_0}_{\ell})^{-1} y_{k_0} & \text{if } k = \ell \\
     0 & k \in \{0,\dots,k_0\}\setminus \{\ell\}
      \end{cases}
    \]
    started at time $k_0$. By construction
    $y_{k_0} = \overline{x}_{k_0}$ and  \eqref{eq:overlinexk-krassno}
    shows that $y_k=\overline{x}_k$ for $k \geq k_0$.
    Since $A$ is Mann-stable, we have $\overline{x}_k \to
      0$ as shown above and therefore also
    $y_k \to 0$.
    This proves Krasnoselskii stability of $A$. \qedhere
 \end{enumerate}
\end{proof}

The previous result together with \Cref{thm:stability-Krasnoselskii}
immediately gives:
\begin{cor}
  Let $A \in \bbC^{d \times d}$.
  The classical Mann iteration of $A$ with
  $\pi^k = (\frac{1}{k+1},\dots,\frac{1}{k+1},0,\dots)$
  is stable if and only if
  $\sigma(A) \subseteq \Hone$.
\end{cor}

\begin{rem}
  \Cref{exam:non-open-stability-region} together
  with \Cref{prop:Krasnoselskii-as-mann} shows
  that there are Mann-weight sequences
  for which Mann stability of a matrix $A$
  cannot be determined from the spectrum of $A$ alone.
\end{rem}

\section{Moving-Average Iteration}
\label{sec:moving-average-iteration}

We have now reached the core of the paper.
In this and the following three sections we determine the region
of stability for the moving-average iteration scheme with window width $m \in \bbN$ given by
\begin{align}
  \label{eq:recall-moving-average}
  x_{k+1} = A\left(\frac{1}{m} \sum_{\ell=0}^{m-1} x_{k-\ell}\right).
\end{align}
In the notation of the previous section
this corresponds to Mann iteration with
weight sequence
\begin{align*}
  \pi^{k} = \Bigl(\underbrace{0,\dots,0 \vphantom{\frac{1}{m}}}_{k+1-m},\underbrace{\frac{1}{m},\dots,\frac{1}{m}}_{m},0,\dots\Bigr).
\end{align*}
Note that 
the iteration in \eqref{eq:recall-moving-average} is only defined for $k\geq m-1$.
However, in light of \Cref{rem:mann-only-eventual-weights-are-relevant}
it is irrelevant for our stability considerations how
$\pi^0,\dots,\pi^{m-2}$ are defined.

\begin{defn}
  We say that $A \in \bbC^{d \times d}$ is \emph{moving-average-stable}
  for window width $m$
  if it is Mann-stable with respect
  to the Mann-weight sequence
  \[\pi^k = \Bigl(\underbrace{0,\dots,0\vphantom{\frac{a}{b}}}_{k+1-m},\underbrace{\frac{1}{m},\dots,\frac{1}{m}}_m,0,\dots\Bigr),\]
  i.e., if and only if every sequence $(x_k)_{k\in\bbN_0}$ in $\bbC^d$  which eventually satisfies \eqref{eq:recall-moving-average}
  converges to zero.
\end{defn}

We begin by making the autonomous structure of the system
 \eqref{eq:recall-moving-average} explicit.

For $A \in \bbC^{d \times d}$ and $m \in \bbN$ define $\calM_m(A)\in \bbC^{(md) \times (md)}$ by
\begin{align}
\calM_m(A):=  \begin{pmatrix}
    0 & I_d & 0 & \dots &0 \\
    0& 0& I_d & \dots &0 \\
    \vdots &\vdots&&\ddots \\
    0&0&0&&I_d \\
    \frac{1}{m}A&\frac{1}{m}A&\frac{1}{m}A&\dots&\frac{1}{m}A
  \end{pmatrix}.
  \label{eq:def-ma}
\end{align}
Define
\begin{align*}
  \mars_m &:= \{a \in \bbC \setsep \sigma(\calM_m(a)) \subseteq B_{1}(0)
  \},\\
  \mars &:= \bigcup_{m \in \bbN} \mars_m.
\end{align*}
We call $\mars$ and $\mars_m$ the \emph{moving-average region of
  stability  (MARS)} (for window width $m$). 
  For an illustration of the relationship between $a, \mars_m$
and $\sigma(\calM_m(a))$ see \Cref{fig:spectra}.
\begin{figure}
  \begin{center}
\includegraphics[width=\textwidth]{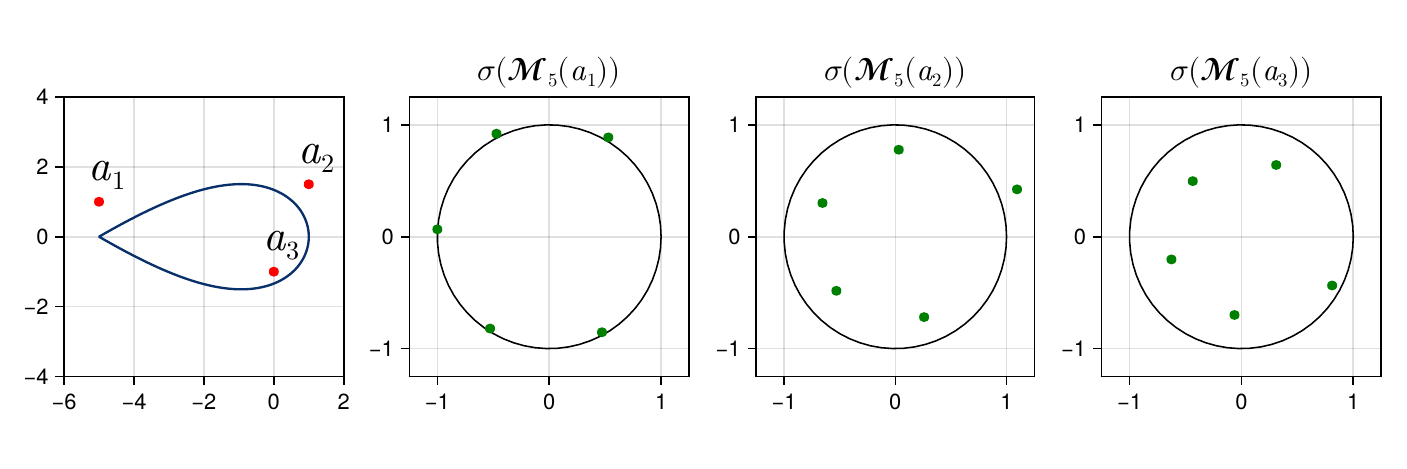}
  \vspace{-1.4cm}
  \end{center}
  \caption{The first picture shows three points $a_1, a_2,a_3$ in relation to $\mars_5$. The corresponding spectra
    of $\calM_5(a_\ell)$ are shown in the following three figures.}
  \label{fig:spectra}
\end{figure}
\begin{thm}
\label{thm:mas-via-mars}
  For $A \in \bbC^{d \times d}$ and $m \in \bbN$ the following are
  equivalent:
  \begin{enumerate}[(i)]
  \item $A$ is moving-average-stable
    for window width $m$.
  \item $\sigma(\calM_m(A)) \subseteq B_1(0)$.
  \item $\sigma(A) \subseteq \mars_m$.
  \end{enumerate}
\end{thm}
\begin{proof}
  The equivalence between (i) and (ii)
  is an immediate consequence of the following observation:
  If $(y_k)_{k \in \bbN_0}$ is the solution of
  moving-average iteration of $A$ with
  window width $m$ started
  at time $k_0 \geq m-1$ with initial values
  $y_0,\dots,y_{k_0} \in \bbC^d$,
  then
  \begin{align*}
    (y_{k-m+1}^\top,\dots,y_{k}^\top)^\top=\calM_m(A)(y_{k-m}^\top,\dots,y_{k-1}^\top)^\top \text{ for } k > k_0.
  \end{align*}
  The equivalence between (ii) and (iii) follows directly
  from
  \begin{align*}
    \sigma(\calM_m(A)) = \bigcup_{a \in \sigma(A)} \sigma(\calM_{m}(a)).
  \end{align*}
  To see this, notice that the eigenvectors $v$ of $\calM_m(A)$ corresponding
  to the eigenvalue $\lambda$ are precisely the vectors
  of the form $(\lambda^0,\dots,\lambda^{m-1}) \otimes  w$
  for some $w \in \bbC^d, w\neq 0$, with 
  \begin{align}
    \label{eq:lambda-condition}
    \frac{1}{m} \sum_{\ell=0}^{m-1}
    \lambda^\ell A w = \lambda^m w.
  \end{align}
  Now the existence of $w \neq 0$, for
  which \eqref{eq:lambda-condition} holds, 
  is equivalent to 
  \begin{align*}
    0 \in \sigma\left(\frac{1}{m} \sum_{\ell=0}^{m-1}
    \lambda^\ell A-\lambda^mI \right).
  \end{align*}
  By the spectral mapping theorem for polynomials, this is in  turn
  equivalent to the existence of $a \in \sigma(A)$ with
  \begin{align*}
    \lambda^m = \frac{1}{m} \sum_{\ell=0}^{m-1}
    \lambda^\ell a,
  \end{align*}
  i.e. $\lambda \in \sigma(\calM_m(a))$.
\end{proof}

In the following, we will often restrict attention
to $a$ in the open upper half-plane. We therefore define
\begin{align*}
  \mars_m^+ &:= \{a \in \bbC \setsep a \in \mars_m,\; \imag(a)>0 \},\\
  \mars_m^- &:= \{a \in \bbC \setsep a \in \mars_m,\; \imag(a)<0 \}.
\end{align*}

In the one-dimensional case $\calM_m(a)$ is a companion matrix,
so its characteristic polynomial can be directly read off as
\begin{equation}
 p_{m,a}(s) := s^m - \frac{a}{m}\sum_{\ell=0}^{m-1} s^\ell = s^m - \frac{a}{m}
  \frac{s^m-1}{s-1}.
  \label{eq:pma}    
\end{equation}

Multiplication by $(s-1)$ gives
\begin{align}
  q_{m,a}(s) := (s-1)p_{m,a}(s) = s^{m+1} -\left(1+\frac{a}{m}\right) s^m + \frac{a}{m}.
  \label{eq:qma}
\end{align}

The companion matrix of the polynomial $q_{m,a}(s)$ is 
  \begin{align*}
    \calQ_{m}(a):=\begin{pmatrix}
    0 & 1 & 0 & \cdots &0&0 \\
    0& 0& 1 & \cdots &0&0 \\
    \vdots &\vdots&&\ddots &&\vdots\\
    0&0&0&&1&0 \\
    0&0&0&\cdots&0&1 \\
    -\frac{a}{m}&0&0&\cdots&0&1+\frac{a}{m}.
  \end{pmatrix}                            
  \end{align*}
and the corresponding linear recurrence relation is
\begin{align}
  x_{k+1}=\left(1+\frac{a}{m}\right) x_k -\frac{a}{m} x_{k-m}.
  \label{eq:sparse-recurrence-equation}
\end{align}
The asymptotic stability of such delay difference equations
has received significant attention in the literature.
The following is a chronological overview of increasingly general settings in which necessary and/or sufficient stability conditions have been found:
\begin{align*}
  x_{k+1}&=x_k - \beta x_{k-m} && \beta \in \bbR&& \text{  \cite{levinNoteDifferencedelayEquations1976}};\\
  x_{k+1} &= \alpha x_k - \beta x_{k-m} && \alpha, \beta \in\bbR&&
\text{
\cite{kuruklisAsymptoticStabilityXn1994}};\\
  x_{k+1} &= \alpha x_k + Bx_{k-m} && \alpha \in \{+1,-1\}, B \in \bbR^{d \times d}&& 
  \text{
\cite{kipnisStabilityDelayDifference2006}};\\
  x_{k+1} &= \alpha x_k + Bx_{k-m} && \alpha \in (0,1), B \in \bbR^{d \times d}&& 
  \text{
\cite{kaslikStabilityResultsClass2009}}; \\
   x_{k+1} &= \alpha x_k + \beta x_{k-m} && \alpha, \beta \in \bbC&& 
   \text{
\cite{cermakStabilitySwitchesLinear2014}}; \\        
    x_{k+1} &= A x_k + Bx_{k-m} && A, B \in \bbR^{d \times d}&& 
    \text{
\cite{diblikj.StabilityExponentialStability2015}}. 
\end{align*}
Note, however, that \eqref{eq:sparse-recurrence-equation}
is never asymptotically stable in $0$, since every constant sequence is a solution.
This is due to the fact that we added $\lambda=1$ as a root 
when we passed from \eqref{eq:pma} to \eqref{eq:qma}.
Hence, while \cite{cermakStabilitySwitchesLinear2014} covers our
equation, the corresponding stability results are not 
applicable to our case.

Summing up the previous consideration we obtain
\begin{prop}[Characterization of $\mars_m$ via roots of polynomials]
  \begin{align*}
    \mars_m &= \{a \in \bbC \setsep p_{m,a} \text{ has all roots  in }
    B_1(0)\} \\
    &= \{a \in \bbC \setsep q_{m,a} \text{ has all but one root in }
      B_1(0)\}.
  \end{align*}
\end{prop}

By Perron-Frobenius theory, the Gauß-Lucas theorem, and some elementary real algebraic geometry we obtain the following.
\begin{prop}[Simple properties of $\mars_m$ and $\mars$]
  \label{prop:simple-properties}
  For $m \in \bbN$ we have
  \begin{enumerate}[(a)]    
    \item $\mars_m \cap [0,\infty) = [0,1)$.\label{item:mars-reals-right}
    \item $(-m,0) \subseteq \mars_m$, hence $(-\infty,0) \subseteq \mars$.\label{item:mars-real-left}
    \item $\mars \subseteq \Hone$.\label{item:mars-is-left-of-one}
    \item $\mars_m^-=(\mars_m^+)^*:= \{ \overline{a} \setsep a \in \mars_m^+\}$. 
      \label{item:decomp-mars}
    \item $\mars_m$ is open. \label{item:mars-is-open}
    \item $\mars_m$ is semi-algebraic\footnote{We refer to \cite{bochnak2013real} for fundamentals of real algebraic geometry. We use the standard identification of $\bbC$ with $\bbR^2$ to refer to subsets of $\bbC$ as semi-algebraic, if their real representation in $\bbR^2$ is.} in $\bbC$ and $\mars$ is semi-analytic but not semi-algebraic.
    \label{item:mars-is-semialgebraic}
  \end{enumerate}
\end{prop}
\begin{proof}
  Applying Perron-Frobenius theory directly to $\calM_m(a)$ gives us
  (a). Indeed, for $a>0$ the matrix
  $\calM_m(a)$ is nonnegative and primitive.
  Therefore $\rho(\calM_m(1))=1$, as we have an eigenvector
  of all ones in this case.
  Now for $a>1$, $\calM_m(a) \gneqq \calM_m(1)$ and hence
  $\rho(\calM_m(a)) > \rho(\calM_m(1)) = 1$, and on the other hand for
  $a<1$ we have
  $\calM_m(a) \lneqq \calM_m(1)$ and so $\rho(\calM_m(a)) < \rho(\calM_m(1)) = 1$.

  For (b) consider the matrix $\calQ_m(a)$.
  If $a \in(-m,0)$, then this matrix is nonnegative and primitive
  and since the all ones vector is a strictly positive
  eigenvector for the eigenvalue $1$, it has spectral radius $1$.

  Hence, by the Perron-Frobenius theorem, for $a \in (-m,0)$, the polynomial $q_{m,a}$
  has a simple root at $1$ and all other roots have
  absolute value smaller than $1$. This shows $(-m,0) \subseteq \mars_{m}$.

  For (c) consider $a \in \bbC$ with $a\neq 1$ and $\real a\geq 1$.
  The polynomial \[q_{m,a}'(s)=\frac{d}{ds} q_{m,a}(s) = (m+1)s^m - (m+a)s^{m-1}\]
  has a root of multiplicity $m-1$ at $0$ and
  a simple root at $\frac{m+a}{m+1}$.
  Since by the Gauß-Lucas theorem
  the roots of $q_{m,a}'$ are contained
  in the convex hull of the roots of $q_{m,a}$,
  there exists a root of $q_{m,a}$ with
  absolute value at least $\abs{\frac{m+a}{m+1}}$.
  We assume $\real a \geq 1$, hence $\abs{a-(-m)} \geq \abs{1-(-m)}$
  for all nonnegative $m$. Therefore $\mars \subseteq \Hone$.

  We have $\sigma(\calM_m(\overline{a}))=\sigma(\overline{\calM_m(a)})=\sigma(\calM_m(a))^*$.
  Combined with the fact that  $B_1(0)$ is invariant under conjugation, the set $\mars_m$ is
  invariant under conjugation as well. This shows (d).

  For (e) let $X$ be the space of
  closed subsets of $\bbC$ endowed with the Hausdorff topology.
  Notice that $\sigma(\calM_m(a)) \in X$ depends
  continuously on $a$ which follows, for example, from Elsner's theorem,
  see e.g. \cite[Theorem 1.3]{stewartMatrixPerturbationTheory1990}.
  Since $\{M \subseteq B_1(0) \setsep M \text{ is closed}\}$ is open
  in $X$, the set $\mars_m$ is open as well.

  The claim (f) is a standard application of the Tarski-Seidenberg theorem on quantifier elimination, \cite[Theorem 1.4.2]{bochnak2013real}.
  Using this result, we first note that the following set is semialgebraic:
  \begin{align*}
      K := \{ (\xrep, \yimp, \alpha,\beta) \in \bbR^4 \setsep  &\exists (v,w) \in \bbR^{2md} : (v,w) \neq (0,0) \ \wedge \\
      &\calM_m(\xrep + \iu \yimp) (v+\iu w) =  (\alpha + \iu \beta) (v+\iu w) \}. 
  \end{align*}
  $\mars_m$ is the complement of the set
  \begin{equation*}
     \left\{ \xrep + \iu \yimp \in \bbC \setsep   \exists (\alpha,\beta) \in \bbR^2 : (\xrep, \yimp, \alpha,\beta) \in K \wedge \alpha^2 + \beta^2 \geq 1 \right\} .
  \end{equation*}
  The latter set is again semialgebraic by the Tarski-Seidenberg theorem and as complements of semi-algebraic sets are semi-algebraic the result for $\mars_m$ follows. The claim for $\mars$ will follow from the concrete description obtained in \Cref{cor:explicitformulaasymptoticmars}. 
\end{proof}

At this point one might suspect that $\mars$ just equals $\Hone$.
To our surprise, this turned out not to be the case.
Using the Schur-Cohn method, we will see in the following sections that $\mars$ is the proper subset
of $\Hone$ depicted in \Cref{fig:stability-regions}.

\section{Explicit Formulas for MARS Using Schur-Cohn}
\label{sec:Schur-Cohn}

In this section we derive explicit formulas for $\mars_m$ and $\mars$. Our main tool will be the classical Schur-Cohn characterization of Schur stability of matrices. We begin by recalling this criterion. A brief discussion about the characterization of positive semidefiniteness of Hermitian matrices related to our case will be required until we get to the main result of this section.

A version of the Schur-Cohn theorem applicable to our use case can be found in
\cite[Theorem 3.4.89]{hinrichsenMathematicalSystemsTheory2010}. To state it, we need some additional notation.  For a polynomial
\begin{equation}
\label{eq:Schur-Cohn-poly}
    p(s) =  \sum_{\ell=0}^n a_\ell s^{\ell}  \in \bbC[s]
\end{equation}
we denote by
$p^*(s) := \sum_{\ell=0}^n \overline{a_\ell}s^{n-\ell} 
= s^n \overline{p}(1/s)$
the \emph{Schur-reflection of $p$ of order $n$}, see \cite[Definition
3.4.77]{hinrichsenMathematicalSystemsTheory2010}. Note that $\lambda \neq 0$ is a root of $p$ if and only if
$({\overline{\lambda}})^{-1}$
is a root of $p^*$. 
Given $p(s)$ as in \eqref{eq:Schur-Cohn-poly} and following
\cite[Definition 3.4.85]{hinrichsenMathematicalSystemsTheory2010}, we define matrices
\begin{equation*}
    U(p) := \begin{bmatrix}
        \overline{a}_n & \ldots &\overline{a}_{2} & \overline{a}_1 \\
        0 &\ddots& \ddots & \overline{a}_2\\
        \vdots& \ddots & \ddots &\vdots\\
        0 & \ldots &0& \overline{a}_{n}
    \end{bmatrix}, \quad  L(p):= \begin{bmatrix}
        0 & \ldots &0 & a_0 \\
        \vdots && \iddots & a_1\\
        0& \iddots & \iddots &\vdots\\
        a_0 & a_1 &\ldots & a_{n-1}
    \end{bmatrix},
\end{equation*}
and the Schur matrix $S(p) := U(p)^* U(p) - L(p)^*L(p)$.

\begin{thm}[Schur-Cohn]
  \label{thm:Schur-Cohn}
  Let $q(s) \in \bbC[s]$ be a polynomial of degree $n$
  with Schur matrix $S(q)$. Let $r$ be the rank of $S(q)$
  and $\sigma$ the signature of $S(q)$, i.e.
  the number of positive eigenvalues (counted with multiplicity)
  minus the number of negative eigenvalues.
  Then the polynomial $q(s)$ has $n-r$ roots (again counted with multiplicity) in common with $q^*(s)$,
  $(r+\sigma)/2$ additional roots with modulus smaller than $1$
  and $(r-\sigma)/2$ additional roots with modulus larger than $1$.
\end{thm}
Applied to $q_{m,a}$, which is a polynomial of degree $m+1$, we obtain:
\begin{cor}
  \label{cor:MARS-via-Schur-matrix}
    $a \in \mars_m$ if and only if $S(q_{m,a})$ has rank $m$ and signature $m$.
\end{cor}
\begin{proof}
    We already know that $a \in \mars_m$
    if and only if $q_{m,a}$ has one root at $s=1$ and all
other roots inside the open unit disk.
Now $q_{m,a}$ and $q^*_{m,a}$
always have $s=1$ as a common root.
If $S(q_{m,a})$ has rank $m$ and signature $m$, 
then by \Cref{thm:Schur-Cohn} both polynomials only have one root in common and
all other roots have modulus smaller than one, so $a \in \mars_m$.

If on the other hand $q_{m,a}$
has a root at one and all other roots
have modulus smaller than one, then 
none of these roots $\lambda$ in the open unit disk
can be a root of $q^*_{m,a}$, for
otherwise ${\overline{\lambda}}^{-1}$
would be a root of $q_{m,a}$ of
modulus larger than one.
Therefore $q_{m,a}$ and $q^*_{m,a}$
have only one root in common,
hence $S(q_{m,a})$
has rank equal to $m=(m+1)-1$
and the rank and signature of 
$S(q_{m,a})$  agree.
\end{proof}

As we just noticed, our Schur matrix $S(q_{m,a})$ is never positive definite,
so we cannot apply the standard leading principal minor test,
also known as Sylvester's criterion, see e.g. 
\cite[Theorem 7.5 (b)]{hornMatrixAnalysis1985}. Nevertheless, we want our Schur matrix to be \enquote{as positive
definite as it can be}, which only makes tiny modifications to the
test necessary. 
The following lemma is what we need. It is probably 
well-known but the only
precise reference we could track down was \cite[Theorem 5]{epperleinMesoscaleObstructionsStability2013}
where the case of symmetric real matrices was treated. Note that in general the test for positive semidefiniteness involves all principal minors of a Hermitian matrix,
not just the leading ones, see \cite[Theorem 7.5 (a)]{hornMatrixAnalysis1985}. See also
\cite{prussingPrincipalMinorTest1986} for a discussion of related errors in the literature. For the sake of completeness 
we give a short proof 
using Cauchy's eigenvalue interlacing
theorem, see \cite[Theorem 4.3.17 and Theorem 7.5 (c)]{hornMatrixAnalysis1985}.

\begin{lem}[Characterization of almost positive definite matrices]
\label{lem:char-almost-pos-def}
  Let $C$ be an $n\times n$ Hermitian matrix with zero row sums. Then
  the following are equivalent:
  \begin{enumerate}[(i)]
  \item  \label{item:apd} $C$ is positive semidefinite and has rank $n-1$. 
    \item  \label{item:submatrix-pd} the leading $(n-1)\times (n-1)$ submatrix of $C$,
      which we denote by $C_{(n-1) \times (n-1)}$, is positive definite,
    \item the first $n-1$ leading principal minors of $C$ are
      positive. \label{item:lead-princ-min}
  \end{enumerate}
\end{lem}
\begin{proof}
  The equivalence between \eqref{item:submatrix-pd} and
  \eqref{item:lead-princ-min}
  is the classical characterization of positive definiteness
  via leading principal minors, \cite[Theorem 7.5 (b)]{hornMatrixAnalysis1985}.

   Let $\lambda_1 \leq \dots \leq \lambda_n$ be the eigenvalues of $C$ and $\mu_1 \leq \dots \leq \mu_{n-1}$ be the eigenvalues of $C_{(n-1) \times (n-1)}$. By Cauchy's interlacing theorem we have
  \begin{equation}
  \label{eq:Cauchy}
      \lambda_1 \leq \mu_1 \leq \lambda_2 \leq \dots \leq \mu_{n-1} \leq \lambda_n.
  \end{equation}
  Now (ii) implies $\mu_1>0$ and thus $\lambda_2>0$ and $\lambda_1 =0$ by the assumption of zero row sums. This implies (i). Conversely, if (i) holds, then \eqref{eq:Cauchy} shows that $C_{(n-1) \times (n-1)}$ is positive semidefinite as $\lambda_1=0$. If $C_{(n-1) \times (n-1)}$ were singular, then we would have
  $w \in \bbC^{n-1}$ such that $w^* C_{(n-1) \times (n-1)} w=0$.
  Extending $w$ by $0$ in the last coordinate, we
  would obtain $\tilde{w} \in \bbC^n$ such that
  $\tilde{w}^* C \tilde{w} = 0$, but $\tilde{w} \not\in
  \linSpan \{(1,\dots,1)\}=\ker C$. Hence $C_{(n-1)\times (n-1)}$ is
  regular and thus positive definite. This shows (ii).
\end{proof}

Combining \Cref{cor:MARS-via-Schur-matrix} with
  \Cref{lem:char-almost-pos-def} we obtain the following explicit formulas for $\mars_m^+$.
\begin{thm}[An explicit formula describing the upper half of MARS]
  \label{thm:explicit-formulas-for-mars}
  \begin{align*}
    \mars_m^+ &= \left\{a \in \bbC \setsep \imag\left(\left(1+\frac{a}{m}\right)^{\ell+1}\overline{a}\right)<0
          \text{ for all } \ell \in \{0,\dots,m\},\; \imag(a)>0\right\} \\
    &=\left\{ \xrep{} + \iu \yimp{} \setsep (\xrep,\yimp) \in \bbR \times (0,\infty),\;
  (m+1)\arccot\left(\frac{\xrep+m}{\yimp}\right) <\arccot\left(\frac{\xrep}{\yimp}\right)\right\},
  \end{align*}
  where we consider $\arccot$ as a continuous decreasing function
from $\bbR$ to $(0,\pi)$.
  See \Cref{fig:stability-regions} for an illustration.
\end{thm}

\begin{proof}
We first calculate the Schur matrix of $q_{m,a}$.
To simplify notation, we write $b_m=1+\frac{a}{m}$,
so we have $q_{m,a}(s) = s^{m+1}-b_m s^m + (b_m-1)$.
We start with the two auxiliary matrices
\begin{align*}
  U(q_{m,a}) &= \begin{bmatrix}
      1 & -\overline{b}_m & 0&\dots &0\\
      0 & 1 &-\overline{b}_m& \ddots & \vdots \\
      \vdots &\ddots&\ddots &\ddots&0\\
      \vdots&&\ddots&1&-\overline{b}_m\\
      0&\dots&\dots&0&1\end{bmatrix}\\
  L(q_{m,a}) &=  \begin{bmatrix}
      0 & \dots &\dots&0&b_m-1 \\
      0 & \dots&&b_m-1&0 \\
      \vdots&&\iddots&\iddots&\vdots\\
      0&b_m-1&0&\dots&0\\
      b_m-1 &0&\dots&0& -b_m
    \end{bmatrix} .
\end{align*}
From these matrices a straightforward calculation yields the Schur matrix as
\begin{align*}
&S(q_{m,a}) = U(q_{m,a})^*U(q_{m,a}) -
                 L(q_{m,a})^*L(q_{m,a}) \\
         =&\begin{bmatrix}
      2 \real b_m -|b_m|^2 & -\overline{b}_m & 0&\dots &(\overline{b}_m-1)b_m\\
      -b_m & 2 \real b_m &-\overline{b}_m& \ddots & \vdots \\
      0 &\ddots&\ddots &\ddots&0\\
      \vdots& \ddots & -b_m&2 \real b_m &-\overline{b}_m\\
      (b_m-1)\overline{b}_m&\dots& 0 &-b_m&2\real b_m-|b_m|^2\end{bmatrix}.
\end{align*}
The matrix $S(q_{m,a})$ has $(1,\dots,1)^\top$ in its kernel,
so in particular it is singular.

In light of \Cref{lem:char-almost-pos-def}, we are interested in those parameters $a$ for which
the first $m$ leading principal minors of $S(q_{m,a})$ are positive.
Let $c_{\ell,m,a}$ be the determinant
of the leading $\ell \times \ell$ submatrix
of $S(q_{m,a})$.
Since the leading $m\times m$ submatrix
is tridiagonal, we obtain the following recursion
\begin{align}
\nonumber
  c_{0,m,a} &= 1,\\
\label{eq:det-recursion}
  c_{1,m,a} &= 2 \real b_m -|b_m|^2, \\
  c_{\ell+2,m,a} &= 2 \real b_m c_{\ell+1,m,a}- |b_m|^2 c_{\ell,m,a}, \quad \ell = 0, \ldots, m-2.
  \nonumber
\end{align}

We assume for the remainder of the proof that $\imag(a)>0$. Then $b_m$
is not real and we get two different roots $s_1=b_m$ and
$s_2=\overline{b}_m$ for the characteristic equation
$s^2=2 \real b_m s- |b_m|^2$ of the recursion
\eqref{eq:det-recursion}.  We therefore obtain
\begin{align*}
  c_{\ell,m,a} &= b_m^{\ell+1} \frac{\overline{b}_m-1}{\overline{b}_m-b_m} + (\overline{b}_m)^{\ell+1}
  \frac{b_m-1}{b_m-\overline{b}_m} \\
               &= \real \left(b_m^{\ell+1}\frac{\iu \overline{a}}{\imag a}\right)=\frac{-\imag(b_m^{\ell+1} \overline{a})}{\imag a}
\end{align*}
for all $\ell \in \{0,\dots,m\}$.

All in all this shows that
\begin{align*}
  \mars_m^+ = \{a \in \bbC \setsep  \imag(b_m^{\ell+1} \overline{a})<0 \text{ for all }\ell \in \{0,\dots,m\}, \imag(a)>0\}.
\end{align*}

It remains to show the second equality in the assertion.
Set $\varphi_{a,m} = \arg{b_m} \in (0,\pi)$.
Note that $\imag(b_m \overline{a}) <0$.
As $\arg \overline{a} \in (-\pi,0)$,
we have $\varphi_{a,m} +\arg \overline{a} =\arg(b_m \overline{a})\in (-\pi,0)$.

Since $\varphi_{a,m} \in (0,\pi)$ the condition
\begin{align}
  \forall \ell \in \{0,\dots,m\}: \imag\left(b_m^{\ell+1} \overline{a}\right)<0
\end{align}
becomes
\begin{align}
  \forall \ell \in \{0,\dots,m\}: (\ell+1) \varphi_{a,m} +\arg \overline{a} \in
  (-\pi,0). \label{eq:arg-condition}
\end{align}
Because   $\varphi_{a,m} +\arg \overline{a} \in (-\pi,0)$
and $\varphi_{a,m}>0$
this is equivalent to 
\begin{align}
\label{eq:phi-final-ineq}
  (m+1)\varphi_{a,m} + \arg \overline{a} < 0.
\end{align}

Set $a=\xrep+\iu\yimp{}$. We have already assumed that $\yimp>0$.
Recall that we consider $\arccot$ as a continuous decreasing function
from $\bbR$ to $(0,\pi)$,
see \Cref{fig:acot}.
\begin{figure}
\begin{center}
\begin{tikzpicture}
\begin{axis}[
    height=4cm,
    width=8cm,
    axis lines=middle,
    xmin=-7, xmax=7.2,
    ymin=-.2, ymax=3.5,
    samples=400,
    xlabel={$\xvar$},
    ylabel={$\arccot(\xvar)$},
    ytick={pi/2,pi},
    yticklabels={$\frac{\pi}{2}$,$\pi$},
    x label style={at={(current axis.right of origin)},anchor=north, below=1mm},
       very thick,
]
\addplot[thick,NavyBlue,domain=-7:7.2
]
{rad(atan2(1,x))};
\end{axis}
\end{tikzpicture}
\caption{The branch of $\arccot$ used throughout the paper}
\label{fig:acot}
\end{center}
\end{figure}
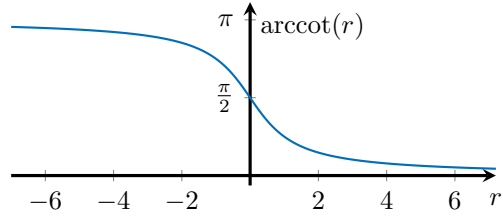
With this convention we get
\begin{align*}
  \varphi_{a,m} = \arccot\left(\frac{\xrep+m}{\yimp}\right)
\end{align*}
and hence condition \eqref{eq:phi-final-ineq} becomes
\begin{align*}
  (m+1)\arccot\left(\frac{\xrep+m}{\yimp}\right) &< - \arg \overline{a} = \arg a=
  \arccot\left(\frac{\xrep}{\yimp}\right). 
\end{align*}
The second equality in the assertion is proved.
\end{proof}

\section{The Structure of MARS}
\label{sec:MARSstructure}

In this section we will exploit the formula obtained
in \Cref{thm:explicit-formulas-for-mars} to obtain a better understanding of
the structure of the sets $\mars_m$.
Our main tool will be the function $F: \bbR \times (0,\infty)
\times [0,\infty) \to \bbR$ given by
\begin{align*}
  F(\xrep,\yimp,m)=(m+1)\arccot\left(\frac{\xrep+m}{\yimp}\right).
\end{align*}
Using this notation \Cref{thm:explicit-formulas-for-mars} reads
\begin{equation}
\label{eq:thm74recap}
  \mars_m^+ = \{\xrep+\iu\yimp{} \setsep (\xrep,\yimp) \in \bbR \times (0,\infty),\; F(\xrep,\yimp,m) <F(\xrep,\yimp,0)\}
\end{equation}
and we will show  in \Cref{cor:decomp-mars} that
\begin{align*}
\mars_m &= \mars_m^+ \cup \mars_m^- \cup (-m,1)\\
&=\{\xrep{}+\iu\yimp{} \setsep \xrep{},\yimp{} \in \bbR, F(\xrep{},\abs{\yimp{}},m) <F(\xrep,\abs{\yimp},0),  \yimp{} \neq
0\} \cup (-m,1).
\end{align*}
We have moved the discussion of various estimates which do not directly involve the function $F$ to
\Cref{sec:various-estimates}.
We start with a simple observation.
\begin{lem}[A quarter plane containing the upper part of MARS]
\label{lem:quarter-plane}
  Let $m \in \bbN$. If $\xrep+\iu\yimp{} \in \mars_m^+$, then
  $1> \xrep{} > -m+ \yimp{} \cot\left(\frac{\pi}{m+1}\right)\geq -m$.
\end{lem}
\begin{proof}
  For $\xrep+\iu\yimp{} \in \mars_m^+$ we have $F(\xrep,\yimp,m)<F(\xrep,\yimp,0)<\pi$ and therefore
  $(m+1)\arccot(\frac{\xrep+m}{\yimp})<\pi$. This in turn implies the first claim through
  \begin{align*}
    \frac{\xrep+m}{\yimp}&>\cot\left(\frac{\pi}{m+1}\right) \geq \cot\left(\frac{\pi}{2}\right)=0. 
  \end{align*}
  The observation $\xrep<1$ is another way of stating \Cref{prop:simple-properties}
  \eqref{item:mars-is-left-of-one}.
\end{proof}

Combining the previous results we
obtain the following decomposition
for the whole moving-average region of stability.
\begin{cor}[Decomposition of MARS] For all $m \in \bbN$ 
  \label{cor:decomp-mars}
  \begin{align*}
    \mars_m &= \mars^+_m \cup \mars_m^- \cup (-m,1) = \mars^+_m \cup \left(\mars_m^+\right)^* \cup (-m,1).
  \end{align*}  
\end{cor}
\begin{proof}
  By \Cref{lem:quarter-plane}
  and \Cref{prop:simple-properties}
  \itemref{item:decomp-mars}
  it only remains to show that
  $\mars_m \cap \bbR = (-m,1)$.
  We already know that $(-m,1) \subseteq \mars_m$
  from \Cref{prop:simple-properties} \itemref{item:mars-reals-right}
  and \itemref{item:mars-real-left}.

  By \Cref{prop:simple-properties}
  \itemref{item:mars-is-open}, $\mars_m$ is open. Thus 
  the inclusion $\mars_m \cap \bbR \subseteq (-m,1)$ follows
  directly from \Cref{lem:quarter-plane}.
\end{proof}

We will now make the semialgebraic structure of $\mars_m$ explicit, which was already obtained in \Cref{prop:simple-properties}
  \itemref{item:mars-is-semialgebraic}.
Using the addition theorem for $\cot$, we can describe
the sets $\mars_m$ as sublevel sets of rational functions, 
which we calculate explicitly for small $m$.

\begin{thm}[Describing MARS by rational functions]
  \label{thm:rational-func-description}
  For every $m \geq 1$ there is a rational function $h_m$ in two variables
  such that \begin{align*}
    \mars_m^+= \left\{\xrep{} + \iu \yimp{} \in \bbC \setsep \yimp{} >0, h_m(\xrep{} ,\yimp{} )<0 \text{ and } \xrep{}  > -m+ \yimp{}  \cot\left(\frac{\pi}{m+1}\right)\right\}.
  \end{align*}
            Furthermore
            \begin{align*}
    \mars_1^+&= \{\xrep{} + \iu \yimp{} \setsep \xrep{} ^2+\yimp{} ^2<1,\yimp{} >0\}, \\
    \mars_2^+&= \{\xrep{} + \iu \yimp{} \setsep (\xrep{} ^2+\yimp{} ^2)(\xrep{} +3)<4, \xrep{} >-2,\yimp{} >0\}, \\
    \mars_3^+&= \{\xrep{} + \iu \yimp{} \setsep
           3(\xrep{} ^2+\yimp{} ^2)(\xrep{} +3)(\xrep{} +5)+(9-\yimp{} ^2)(\xrep{} ^2+\yimp{} ^2)<81,\\ & \hspace*{8cm} \xrep{} >-3+\yimp{} ,\yimp{} >0\}.
            \end{align*}
          \end{thm}

\begin{rem}
\label{rem:euler}
    The crucial tool behind \Cref{thm:rational-func-description} is the fact
  that $\tan(n \beta)$ can be expressed
  as a rational function $R_n$ of $\tan(\beta)$.
  This has been known for a long time, e.g. it appears explicitly
  in Euler's work \cite[Chapter XIV, Item 249]{euler1988}. 
  See \cite{calcut2008tangentchebyshev} and the references therein for an overview
  of the relevant literature.
\end{rem}

\begin{proof}[Proof of \Cref{thm:rational-func-description}]
  Let $z \in \bbR$ and set $\beta = \arccot(z) \in (0,\pi)$. 
  We can view $R_n(z):=\tan(n\arctan(z))$
  as a variation of the Chebyshev polynomials
  which are defined by $\cos(n \arccos(z))$.
  In view of \Cref{rem:euler}, since 
  $\cot$ is the reciprocal of $\tan$,
  we can also express $\cot(n\beta)$ as
  a rational function $Q_n \in \bbR(r)$ of $\cot(\beta)$
  where $Q_n(\xvar) = \frac{1}{R_n(\frac{1}{\xvar})}$.
  Since we need the first few of these functions explicitly
  and the derivation is elementary and short,
  we reproduce a recursive formula for $Q_n$ below.
  
  For $n \in \bbN$ we have 
  \begin{align*}
    \cot(n \beta)
          &= \frac{\cos(n\beta )}{\sin(n\beta)} \\
          &= \iu\frac{(\cos(\beta) + \iu \sin(\beta))^n+(\cos(\beta) -
            \iu \sin(\beta))^n}{(\cos(\beta) + \iu
            \sin(\beta))^n-(\cos(\beta) - \iu \sin(\beta))^n} \\
          &=i \frac{(z+\iu)^n+(z-\iu)^n}{(z+i)^n-(z-\iu)^n}\\
          &=\frac{\real(z+\iu)^n}{\imag(z+\iu)^n}
  \end{align*}
  Now define a sequence of polynomials $p_n,q_n$ recursively by \begin{align*}
    p_1(s) &= s, \\
    q_1(s) &= 1, \\
    p_{n+1}(s) &= sp_n(s) - q_n(s), \\
    q_{n+1}(s) &= sq_n(s) + p_n(s).
  \end{align*}
  The previous calculation shows that
  \begin{align*}
    \cot(n \arccot(z)) = Q_n(z)=\frac{p_n(z)}{q_n(z)}.
  \end{align*}
  We therefore get
  \begin{align*}
    \cot(2\arccot(z)) &= \frac{z^2-1}{2z},\\
    \cot(3\arccot(z)) &=\frac{z^3-3z}{3z^2-1},\\
    \cot(4\arccot(z)) &=\frac{z^4-6z^2+1}{4z^3-4z},\\
    &\vdots
  \end{align*}
  We now set $h_m(\xrep,\yimp):=\frac{\xrep}{\yimp}-\cot((m+1)\arccot(\frac{\xrep+m}{\yimp}))=\frac{\xrep}{\yimp}-Q_{m+1}(\frac{\xrep+m}{\yimp})$
  which is a rational function in $\xrep$ and $\yimp$.
  Using the previous calculations we obtain
  \begin{align*}
    h_1(\xrep,\yimp) &= \frac{\xrep^{2} + \yimp^{2} - 1}{2 \, {\left(\xrep{} + 1\right)} \yimp}, \\
    h_2(\xrep,\yimp) &= \frac{2 \, {\left(\xrep^{3} + \xrep{} \yimp^{2} + 3 \, \xrep^{2} + 3 \,
               \yimp^{2} - 4\right)}}{{\left(3 \, \xrep^{2} - \yimp^{2} + 12 \, \xrep
               + 12\right)} \yimp} \\
     &= \frac{2((\xrep^2+\yimp^2)(\xrep+3)-4)}{{\left(3(\xrep+2)^{2} - \yimp^2              \right)} \yimp}, \\
    h_3(\xrep,\yimp) &= \frac{3 \, \xrep^{4} + 2 \, \xrep^{2} \yimp^{2} - \yimp^{4} + 24 \, \xrep^{3} + 24 \, \xrep \yimp^{2} + 54 \, \xrep^{2} + 54 \, \yimp^{2} - 81}{4 \, {\left(\xrep + \yimp + 3\right)} {\left(\xrep - \yimp + 3\right)} {\left(\xrep + 3\right)} \yimp} \\           &= \frac{3(\xrep+5)(\xrep^2+\yimp^2)(\xrep+3)-(\yimp^2-9)(\xrep^2+\yimp^2)-81} {4((\xrep+3)^2-\yimp^2)(\xrep+3)\yimp}   .
  \end{align*}

  Using \eqref{eq:thm74recap}, we can now show 
  \begin{align}
  \label{eq:uppermars-by-rational-in-proof}
      \mars_m^+= \left\{\xrep{} + \iu \yimp{} \in \bbC \setsep \yimp>0, h_m(\xrep,\yimp)<0 \text{ and } \xrep{} > -m+ \yimp
  \cot\left(\frac{\pi}{m+1}\right)\right\}
  \end{align}
  as follows:
  If $\xrep{} + \iu \yimp{}\in \mars_m^+$, then $\yimp >0$ and $F(\xrep,\yimp,m)<F(\xrep,\yimp,0)$ and as $\cot$ is monotonically decreasing on $(0,\pi)$ we have $\cot\left((m+1)\arccot(\frac{\xrep+m}{\yimp})\right) >
  \frac{\xrep}{\yimp}$. 
  Therefore $h_m(\xrep,\yimp)<0$. The final condition in the set description on the right of \eqref{eq:uppermars-by-rational-in-proof} is satisfied due to \Cref{lem:quarter-plane}. Thus $\xrep{} + \iu \yimp{}$ is contained in the set on the right hand side of \eqref{eq:uppermars-by-rational-in-proof}.
  Conversely, if $\xrep{} + \iu \yimp{}$ is contained in the set on the right hand side of \eqref{eq:uppermars-by-rational-in-proof}, then the final condition implies $F(\xrep,\yimp,m)<\pi$ and then $h_m(\xrep,\yimp)<0$ implies \begin{align*}
    F(\xrep,\yimp,m)=\arccot(\cot(F(\xrep,\yimp,m)))<F(\xrep,\yimp,0). 
  \end{align*}
  This shows $\xrep{} + \iu \yimp{}\in \mars_m^+$.
  
  Using our explicit formulas for $h_1,h_2$ and $h_3$ we derive the
  explicit formulas for $\mars_1^+,\mars_2^+$ and $\mars_3^+$ as follows.
  Since we are interested in $\mars_m^+$ we always assume $\yimp>0$.
  
  $(m=1)$: 
  By \eqref{eq:uppermars-by-rational-in-proof}
  we have
  \begin{align*}
    \mars^+_1 = \{\xrep{} + \iu \yimp{} \in \bbC \setsep \xrep^2+\yimp^2<1,\,\xrep>-1,\,\yimp>0\},  \end{align*}
  where the condition $\xrep>-1$ is clearly redundant.
  
  $(m=2)$: Assume that $\xrep+\iu\yimp{} \in \mars_2^+$. Then $\xrep>-2+\frac{\sqrt{3}}{3}\yimp$,
  hence \[\left(3(\xrep+2)^2-\yimp^2\right)\yimp>0.\] Thus $h_2(\xrep,\yimp) <0$ implies
  $(\xrep^2+\yimp^2)(\xrep+3)<4$.
  On the other hand,
  assume $(\xrep^2+\yimp^2)(\xrep+3)<4$ and $\xrep>-2$.
  Then $\xrep<1$ and by \Cref{lem:R2-bound-1} we have
   \begin{align*}
    3(\xrep+2)^2>\frac{4}{\xrep+3}-\xrep^2>\yimp^2
  \end{align*}
  Therefore $h_2(\xrep,\yimp)<0$
  and $\xrep+\iu\yimp{} \in \mars_2^+$.

  $(m=3)$: Assume that $\xrep>-3+\yimp$.
  Then $(\xrep+3)^2-\yimp^2>0$ and $\xrep+3>0$, hence $h_3(\xrep,\yimp)<0$ if and only if
  \begin{align*}
    3(\xrep^2+\yimp^2)(\xrep+3)(\xrep+5)+(9-\yimp^2)(\xrep^2+\yimp^2)&<81. \qedhere
  \end{align*}
\end{proof}

\begin{prop}[The upper boundary of MARS]
  \label{lem:upper-boundary-g}
  There are continuous functions
  $g_m:[-m,1] \to [0,\infty)$, $m\in\bbN$, with $g_m(-m) = g_m(1) =0$, which are continuously differentiable
  and positive on $(-m,1)$, such that
     \begin{align}
         \mars_m^+ &= \{ \xrep+\iu\yimp{} \in \bbC \setsep \xrep{} \in
                 (-m,1),\; \yimp{} \in (0,g_m(\xrep{})) \} .
         \label{eq:Fset-char}
     \end{align} 
\end{prop}
\begin{proof}
  For fixed $\xrep\in \bbR$ and $m> 0$
  set $f_{m,\xrep}(\yimp):=F(\xrep,\yimp,m)-F(\xrep,\yimp,0)$, 
  see \Cref{fig:f3x}.
  The assertion of \eqref{eq:Fset-char} thus concerns the question for which $(\xrep,\yimp)\in \bbR\times (0,\infty)$ we have $f_{m,\xrep}(\yimp) <0$.
  
  Note first that for all $\xrep\in \bbR$ we have $\lim_{\yimp \to \infty} f_{m,\xrep}(\yimp)= m  \arccot(0)= \frac{m\pi}{2}$ and that for 
  $\lim_{\yimp \to 0^+} f_{m,\xrep}(\yimp)$
   we have
   \begin{equation}
     \label{eq:lim_fmx_y0}
    \lim_{y \to 0} f_{m,\xrep{} }(y) = \left\{ \quad \begin{matrix}
        m\pi &\;\;& \xrep{} \in (-\infty,-m)\\
        \frac{m-1}{2} \pi && x = -m\\
         - \pi && \xrep{}  \in (-m,0) \\
        -\frac{\pi}{2} && \xrep{}  = 0 \\
        0 && \xrep{}  \in (0, \infty)
    \end{matrix}\right.
  \end{equation}
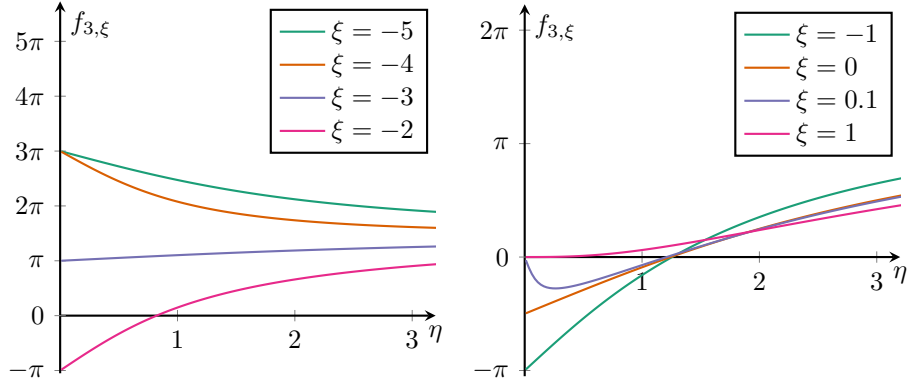
\begin{figure}
\begin{center}
\begin{tikzpicture}
[declare function={F(\x,\y,\k)=%
  (\k+1)*rad(atan2(1,(\x+\k)/\y));
  f(\x,\y,\k)=F(\x,\y,\k)-F(\x,\y,0);}]
\begin{axis}[width=0.54\textwidth,
    height=6.5cm,
    axis lines=middle,
    xmin=0, xmax=3.2,
    ymin=-3.3, ymax=18,
    samples=400,
    legend style={legend cell align=left},
    extra y ticks = {0},
    xlabel={$\yimp$},
    ylabel={$f_{3,\xrep}$},
    ytick={-pi,0,pi,2*pi,3*pi,4*pi,5*pi,6*pi},
    yticklabels={
    $-\pi$,$0$,$\pi$,$2\pi$,$3\pi$,$4\pi$,$5\pi$,$6\pi$},
    x label style={at={(current axis.right of origin)},anchor=north},
    thick,
    domain=0.01:10.2,
    mark=none,
    cycle list/Dark2,
]
\addplot+[mark=none]{f(-5,x,3)};
\addplot+[mark=none]{f(-4,x,3)};
\addplot+[mark=none]{f(-3,x,3)};
\addplot+[mark=none]{f(-2,x,3)};
\legend{$\xrep=-5$,$\xrep=-4$,$\xrep=-3$,$\xrep=-2$}
\end{axis}
\end{tikzpicture}~
\begin{tikzpicture}
[declare function={F(\x,\y,\k)=%
  (\k+1)*rad(atan2(1,(\x+\k)/\y));
  f(\x,\y,\k)=F(\x,\y,\k)-F(\x,\y,0);}]
\begin{axis}[width=0.54\textwidth,
    height=6.5cm,
    axis lines=middle,
    xmin=0, xmax=3.2,
    ymin=-3.3, ymax=7,
    samples=400,
    legend style={legend cell align=left},
    extra y ticks = {0},
    xlabel={$\yimp$},
    ylabel={$f_{3,\xrep}$},
    ytick={-pi,0,pi,2*pi,3*pi},
    yticklabels={
    $-\pi$,$0$,$\pi$,$2\pi$,$3\pi$},
    x label style={at={(current axis.right of origin)},anchor=north},
    thick,
    domain=0.01:10.2,
    mark=none,
    cycle list/Dark2,
]
\addplot+[mark=none]{f(-1,x,3)};
\addplot+[mark=none]{f(0,x,3)};
\addplot+[mark=none]{f(0.1,x,3)};
\addplot+[mark=none]{f(1,x,3)};
\legend{$\xrep=-1$,$\xrep=0$,$\xrep=0.1$,
$\xrep=1$}
\end{axis}
\end{tikzpicture}
\caption{$f_{3,\xrep}$ for various $\xrep$}
\label{fig:f3x}
\end{center}
\end{figure}
  
  We already know from \Cref{lem:quarter-plane} that $\xrep+\iu\yimp{} \in \mars^+_m$ implies $\xrep{}  \in (-m,1)$.
  Within this interval, two phenomena happen.
  For $-m<\xrep{} \leq 0$ the derivative $f_{m,\xrep{} }'$
  is nonnegative. For $\xrep{} \in (0,1)$ the derivative $f_{m,\xrep{} }'$
  is initially negative with exactly one change in sign. In both cases we obtain the existence of a  value $g_m(\xrep{} )$ such that $f_{m,\xrep{} }(\yimp)<0$ precisely on the interval
  $(0,g_m(\xrep{} ))$.
  More precisely, consider
  \begin{align*}
    f_{m,\xrep{} }'(\yimp) = \frac{(m+1)(m+\xrep{} )}{\yimp^2+(m+\xrep{} )^2}-\frac{\xrep{} }{\yimp^2+\xrep{} ^2}.
  \end{align*}
  Now $f_{m,\xrep{} }'(\yimp) \leq 0$, if and only if
  \begin{align*}
    (m+1)(m+\xrep{} )(\yimp^2+\xrep{} ^2) \leq \xrep{} (\yimp^2+(m+\xrep{} )^2)
  \end{align*}
  or, equivalently, if and only if
  \begin{equation}
  \label{eq:moncond}
       (\xrep{} +m+1) \yimp^2 \leq \xrep{} (\xrep{} +m)(1-\xrep{} ).
  \end{equation}
  For $-m<\xrep{} \leq 0$ \eqref{eq:moncond} is not satisfied and so we
  have $f_{m,\xrep{} }' >0$.
  Since in this case $\lim_{\yimp \to 0^{+}} f_{m,\xrep{} }(\yimp)<0$
  and $\lim_{\yimp\to \infty} f_{m,\xrep{} }(\yimp)>0$
  there is a unique value $g_m(\xrep{} )$ such that
  $f_{m,\xrep{}}(\yimp)<0$ if and only if $\yimp < g_m(\xrep{} )$.

  For $\xrep{} \in (0,1)$, the right hand side of
  \eqref{eq:moncond} is positive, and the factor in front of $\yimp^2$ is positive
  as well. Thus there is again a unique constant $g_m(\xrep{} )$ such that the
  inequality holds if and only if $\yimp\leq g_m(\xrep{} )$. In both cases
  this implicitly defines a function $g_m$.

  The arguments so far show \eqref{eq:Fset-char}. It remains to prove
  the properties of the map $g_m:[-m,1]\to [0,\infty)$. For $\xrep{} \in (-m,1)$
  the function $g_m$ is implicitly defined by the equation
  \begin{equation}
      H_m(\xrep{} ,g_m(\xrep{} )) := F(\xrep{} ,g_m(\xrep{} ),m) - F(\xrep{} ,g_m(\xrep{} ),0) = 0.
  \end{equation}
  By the previous considerations we already know that for each $\xrep{} $ the
  solution of this equation is unique in the interval
    $(0,\infty)$. Also for $\xrep{}  \in (-m,1)$ we have
\begin{equation}
    \frac{\partial}{\partial \yimp} H_m(\xrep{} ,g_m(\xrep{} )) =
    f'_{m,\xrep{} }(g_m(\xrep{} )) >0,
    \label{eq:Hm}
\end{equation}
so by the implicit function theorem
the function
$g_m : (-m,1) \to (0,\infty)$ is continuously differentiable. 

Now we investigate
the behavior of $g_m$ at the boundary points of $(-m,1)$.
Note that
\begin{align}
  0<g_m(\xrep{} ) \leq \frac{\xrep{} +m}{\cot(\frac{\pi}{m+1})}
  \label{eq:bounds-for-g_m}
\end{align}
for $\xrep{}  \in (-m,1)$, see \Cref{lem:quarter-plane}.
Thus clearly $\lim_{\xrep{}  \to (-m)^+} g_m(\xrep{} ) =0$.

Now assume $\lim_{\xrep{}  \to 1^-} g_m(\xrep{} ) \neq 0$.
From \eqref{eq:bounds-for-g_m}
we can derive by compactness 
the existence of a sequence $(\xrep{} _n)_{n \in \bbN}$
in $(-m,1)$ converging to $1$
such that $g_m(\xrep{} _n) \to z>0$ for $n \to \infty$.
By the continuity of $g_m$ together with 
\eqref{eq:Hm} we would have $f_{m,1}(z)=F(1,z,m)-F(1,z,0)=0$.
On the other hand, \eqref{eq:lim_fmx_y0}
together with \eqref{eq:moncond} shows that
$f_{m,1}$ is positive on $(0,\infty)$,
contradiction.

Therefore $g_m$ extends continuously to $[-m,1]$ by setting $g_m(-m) = 0=g_m(1)$.
\end{proof}
Using a geometric argument we will now
  describe another useful outer approximation of $\mars_m$.
\begin{prop}[A circumscribed circle for MARS]
  \label{prop:bounding-circle}
  For $m \in \bbN$ we have
  \begin{align*}
    \mars^+_m \subseteq \left\{\xrep+\iu\yimp{} \in \bbC \setsep \left(\xrep-\frac{1-m}{2}\right)^2+\yimp^2 \leq \frac{(m+1)^2}{4}\right\} = B_{\tfrac{m+1}{2}}\left( \frac{1-m}{2}\right).
  \end{align*}
\end{prop}
\begin{proof}
  Let $A=-m\in \bbC$ and $B=1 \in \bbC$. Then the bounding circle $\mathcal{C}_m$ we
  are interested in has diameter $\overline{AB}$.
  By \Cref{thm:explicit-formulas-for-mars} and \Cref{lem:upper-boundary-g} 
  it is enough to show that
  $(m+1)\arccot\left(\frac{\xrep+m}{\yimp}\right) \geq \arccot\left(\frac{\xrep}{\yimp}\right)$ 
  for $C=\xrep+\iu\yimp{} \in \mathcal{C}_m$ with $\yimp>0$.
  Note that %${\color{magenta} weg?:\arctan(\frac{\xrep+m}{\yimp})=\angle CAB}$ \fabian{
  $\arccot(\frac{\xrep+m}{\yimp})=\angle CAB$
  and 
  $\arccot(\frac{\xrep}{\yimp})=\angle CDB$ where $D=0+0\iu \in \bbC$.
  By Thales's theorem the triangle $ABC$ has a right angle at $C$
  and $\abs{DA}=m\abs{BD}$. Therefore the assertion follows
  by \Cref{lem:right-angle-split-m}.
\end{proof}

\begin{rem}
    Since $\sigma(\calM_m(a))$
    depends continuously on $a$,
    we know that 
    on the boundary of $\mars_m$
    the polynomial $p_{m,a}$
    must have a root with 
    absolute value $1$.
    We can solve $p_{m,a}(s) = 0$ in \eqref{eq:pma}
    for $a$ and obtain
    \begin{align*}
    a = \frac{m(s^{m+1}-s^{m})}{s^m-1}.
    \end{align*}
    Using the standard parametrization of the unit circle,
    we obtain that the boundary of $\mars_m$,
    and in particular the graph of $g_m$ 
    considered as a subset of $\bbC = \bbR^2$, must
    be contained in the closure
    of the curves
    \begin{align}
        \alpha \mapsto \frac{m(e^{i\alpha(m+1)}-e^{i\alpha m})}{e^{i\alpha m}-1},
        \alpha \in \left(\frac{2\pi\ell}{m},\frac{2\pi(\ell+1)}{m}\right),\;
        \ell \in \{0,\dots,m-1\}.
        \label{eq:boundary-locus}
    \end{align}
    \begin{figure}
        \begin{center}
        \includegraphics[width=8cm]{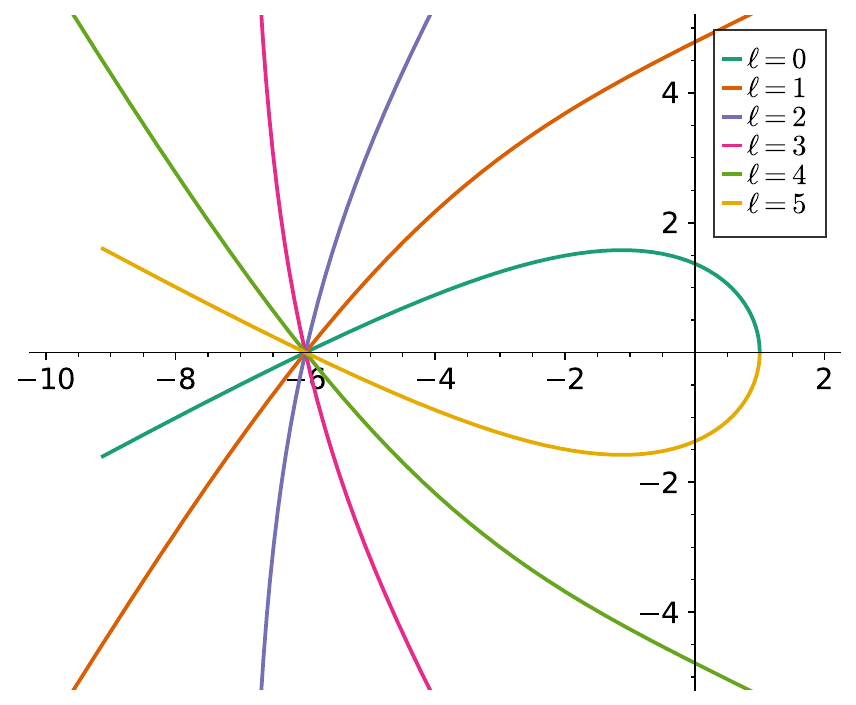}
        \end{center}
        \caption{The curves from \eqref{eq:boundary-locus}
        for $m=6$}
        \label{fig:spider}
    \end{figure}
    See \Cref{fig:spider} for an illustration. 
    This procedure for
    finding the boundary of $\mars_m$ 
    is known in the literature on the numerics of 
    differential equations as the 
    \emph{boundary locus method},
    see \cite[Section 21.4]{salgadoClassicalNumericalAnalysis2022}.
    While
    we know that $\mars_m$ must be a bounded region
    in the complement of these curves, the methods
    used in the remainder of this paper seem
    to be better suited to derive properties
    of $\mars_m$.
\end{rem}

\section{Monotonicity of MARS}
\label{sec:monotonicity}
In this section
we will show that the open sets $(\mars_m)_{m \in \bbN}$ form
an increasing sequence and we will describe the limit set $\mars$ explicitly.

We start by describing a uniform superset of $\mars_m^+$
for all windows sizes $m \in \bbN$.
\begin{lem}[Uniform Bounds for MARS] 
  \label{lem:y-coty}
  If $\xrep+\iu\yimp{} \in \mars_m^+$ for some $m \in \bbN$,
  then $\yimp{} <\pi$ and $\xrep{} < \yimp\cot(\yimp{})$.
\end{lem}
\begin{proof}
  We start by showing
  that $F(\yimp{}\cot(\yimp{}),\yimp{},m)\geq F(\yimp{}\cot(\yimp{}),\yimp{},0)$
  for $\yimp{} \in (0,\pi)$.
  Substituting the definition of $F$ we thus want to show
  \begin{align*}
    \yimp{} \in (0,\pi) \quad \Rightarrow \quad (m+1) \arccot\left(\cot(\yimp{}) + \frac{m}{\yimp{}}\right) \geq \yimp{}
  \end{align*}
  or equivalently
  \begin{align*}
    \yimp{} \in (0,\pi) \quad \Rightarrow \quad \cot(\yimp{}) + \frac{m}{\yimp{}} \leq \cot\left(\frac{\yimp{}}{m+1}\right).
  \end{align*}
  A further equivalent formulation of this statement is
  \begin{align*}
    \yimp{} \in (0,\pi) \quad \Rightarrow \quad m \leq \yimp{}\left(\cot\left(\frac{\yimp{}}{m+1}\right)-\cot(\yimp{})\right).
  \end{align*}
  This claim follows from
  \Cref{lem:lower-bound-fm}.

  Now let $\xrep+\iu\yimp{} \in \mars^+_m$.
  By \Cref{lem:quarter-plane} we know that $\xrep{} \in (-m,1)$. There is $\tilde{\yimp{}} \in (0,\pi)$
  with $\xrep=\tilde{\yimp{}} \cot(\tilde{\yimp{}})$ by \Cref{lem:ycoty-bijective}.
  By the previous argument 
  together with \Cref{thm:explicit-formulas-for-mars}, we have
  $\xrep+\iu\tilde{\yimp} \not \in \mars^+_m$. Together with
  \Cref{lem:upper-boundary-g}
  this shows that $\yimp{} < \tilde{\yimp{}}<\pi$
  and hence $\xrep{} =\tilde{\yimp{}}\cot(\tilde{\yimp{}}) < \yimp{} \cot(\yimp{})$.
\end{proof}

\begin{lem}[$F$ is monotonically decreasing in $m$]
  \label{lem:derivative-m}
  For every $m_0\geq 1$ and all $\xrep+\iu\yimp{} \in \mars_{m_0}^+$, we have
  $\frac{\partial}{\partial m} F(\xrep,\yimp,m)< 0$ for $m > m_0$.
\end{lem}
\begin{proof}
Fix $m_0 \in \bbN$ and $\xrep+\iu\yimp{} \in \mars_{m_0}^+$.
  Define $f: (m_0,\infty) \to \bbR$
  by $f(m)=F(\xrep{},\yimp{},m)$.
  We have $f'(m) = \arccot(\frac{\xrep{}+m}{\yimp{}})-\frac{\yimp{}(m+1)}{\yimp{}^2+(\xrep{}+m)^2}$, $m\in (m_0,\infty)$.
  Hence $f'(m)< 0$
  if and only if \[\arccot\left(\frac{\xrep{}+m}{\yimp{}}\right) < \frac{\yimp(m+1)}{\yimp{}^2+(\xrep{}+m)^2}.\]

For $r>0$ we have $\arctan(r)< r$, hence
  $\arccot(r)=\arctan(\frac{1}{r})< \frac{1}{r}$.  Since $\frac{\xrep{}+m}{\yimp{}}>0$
  it is enough to show that for $m \geq m_0$ we have
\begin{align*}
  \frac{\yimp{}}{\xrep{}+m} \leq \frac{\yimp(m+1)}{\yimp{}^2+(\xrep{}+m)^2}.
\end{align*}
This holds if and only if
\begin{align*}
  \yimp{}^2+(\xrep{}+m)^2 \leq (m+1)(\xrep{}+m).
\end{align*}
This in turn is equivalent to
\begin{align*}
  \yimp{}^2+\left(\xrep{}-\frac{1-m}{2}\right)^2 \leq \frac{(1+m)^2}{4}
\end{align*}
which follows directly from \Cref{prop:bounding-circle}.
\end{proof}

\begin{thm}[MARS is rising]
  \label{thm:monotonicity-mars}
  The sets $\mars_m$, $m\in\bbN$, form an increasing sequence of open sets.
\end{thm}
\begin{proof}
  Let $m \in \bbN$ and $\xrep+\iu\yimp{} \in \mars_{m}^+$.
  By \Cref{lem:derivative-m} we know that
  \[F(\xrep{},\yimp,m+1) < F(\xrep{},\yimp,m) < F(\xrep{},\yimp,0),\] so
  $\xrep+\iu\yimp{} \in \mars_{m+1}^+$.
  Therefore $\mars_{m}^+ \subseteq \mars_{m+1}^+$
  and hence $\mars_{m} \subseteq \mars_{m+1}$.
\end{proof}

\begin{cor}[Explicit formulas for asymptotic MARS] 
\label{cor:explicitformulaasymptoticmars}
We have the
  following formulas for asymptotic MARS
  \begin{align*}
    \mars^+ := \bigcup_{ m \in \bbN} \mars_m^+ 
    &=\left\{ \xrep+\iu\yimp{} \in \bbC \setsep \yimp>0,\; \yimp <
              \arccot\left(\frac{\xrep}{\yimp}\right)\right\} \\
    &=\{a \in \bbC \setsep \imag(a) \in (0,\pi), \imag\left(e^a \overline{a}\right)<0 \}, \\
    \mars  &= \{\xrep{} + \iu \yimp{} \setsep  \xrep<\yimp \cot(\yimp),\; \xrep\in \bbR, \yimp \in 
          (-\pi,0) \cup (0,\pi)\}
          \cup (-\infty,1).
  \end{align*}
\end{cor}
\begin{proof}
  \Cref{cor:decomp-mars} implies that $\mars \cap \bbR = (-\infty,1)$.
  Now by \Cref{thm:monotonicity-mars} and \Cref{thm:explicit-formulas-for-mars}
  we see that 
  \[\mars^+ = \left\{ \xrep+\iu\yimp{} \in \bbC \setsep \yimp>0,\;
    \lim_{m \to \infty} (m+1)\arccot\left(\frac{\xrep+m}{\yimp}\right) \leq \arccot \frac{\xrep}{\yimp}\right\}.\]
   L'Hôpital's rule gives for $\yimp>0$
    \begin{align*}
      \lim_{m \to \infty} (m+1)\arccot\left(\frac{\xrep+m}{\yimp}\right)
      %&= \lim_{m \to \infty} \frac{\frac{1}{\yimp
      %  \left(\frac{(m+\xrep)^2}{\yimp^2}+1\right)}}{\frac{1}{(m+1)^2}}\\
      &= \lim_{m \to \infty} \frac{(m+1)^2}{\yimp
        \left(\frac{(m+\xrep)^2}{\yimp^2}+1\right)} \\
      &=\yimp.
    \end{align*}
    Therefore
    \begin{align*}
      \mars^+ &= \left\{ \xrep+\iu\yimp{} \in \bbC \setsep \yimp>0,\; \yimp <
      \arccot\left(\frac{\xrep}{\yimp}\right)\right\} \\
              &=\left\{ \xrep+\iu\yimp{} \in \bbC \setsep \yimp \in (0,\pi),\; \xrep<\yimp\cot(\yimp)\right\}.
    \end{align*}
    Finally, for $a=\xrep+\iu\yimp{}$ we obtain
    \begin{align*}
      \imag(e^a\overline{a})&=\imag (e^\xrep(\cos \yimp + \iu \sin \yimp)(\xrep-\iu \yimp)) \\
                            &= e^\xrep(-\yimp\cos(\yimp)+\xrep \sin(\yimp)), 
    \end{align*}
    which is negative if and only if $\xrep< \yimp\cot(\yimp)$ for $\yimp \in (0,\pi)$.
  \end{proof}

\section{Convexity of MARS}
\label{sec:convexity}

Our goal in this section is to show that $\mars_m$ is convex for every
$m \in \bbN$.
We start by deriving a representation of $m+\mars^+_m$ in polar
coordinates.
\begin{prop}[MARS in polar coordinates]
\label{prop:mars-in-polar}
  For every $m \in \bbN$ we have
  \begin{align*}
    \mars_m^+ = \left\{-m + r e^{\iu
    \alpha} \setsep \alpha \in \left(0,\frac{\pi}{m+1}\right),\; r \in \left(0,m \frac{\sin((m+1)\alpha)}{\sin(m\alpha)}\right)\right\}.
  \end{align*}
\end{prop}
\begin{proof}
    Let $\xrep + \iu \yimp \in \mars_m^+$.
  Set $\alpha: = \arccot(\frac{\xrep + m}{\yimp}) <
  \frac{1}{m+1}\arccot(\frac{\xrep}{\yimp})$, which implies
  $\alpha \in (0,\frac{\pi}{m+1})$.
  By \Cref{lem:cotm-cotm} we have
  \begin{align*}
    \frac{m}{\yimp} = \frac{\xrep+m}{\yimp} - \frac{\xrep}{\yimp} > \cot(\alpha) - \cot((m+1)\alpha)
    &= \frac{\sin(m \alpha)}{\sin(\alpha) \sin((m+1)\alpha)}
  \end{align*}
  Hence
  \begin{align*}
    (\xrep+m)^2+\yimp^2 = \yimp^2(1+\cot^2(\alpha)) < m^2 \frac{\sin^2((m+1)\alpha)}{\sin^2(m\alpha)}.
  \end{align*}
  This shows
  $\xrep+m+\iu \yimp \in \{ re^{\iu \alpha} \setsep \alpha \in
  (0,\frac{\pi}{m+1}),\; r < m
  \frac{\sin((m+1)\alpha)}{\sin(m\alpha)}\}$.

  On the other hand, let $\alpha \in (0,\frac{\pi}{m+1})$, $r \in(0,m\frac{\sin((m+1)\alpha)}{\sin(m\alpha)})$.
  Set
  \begin{align*}
    \xrep &:= -m+r\cos(\alpha),\\
    \yimp &:=r\sin(\alpha).
  \end{align*}
  Again \Cref{lem:cotm-cotm} implies
  \begin{align*}
    \frac{\xrep}{\yimp} &= \cot(\alpha)-\frac{m}{r\sin(\alpha)}
    &<\cot(\alpha)-(\cot(\alpha)-\cot((m+1)\alpha)=\cot((m+1)\alpha)),
  \end{align*}
  and hence
  \begin{align*}
    \arccot\left(\frac{\xrep+m}{\yimp}\right) &= \alpha < \frac1{m+1} \arccot\left(\frac{\xrep}{\yimp}\right).
  \end{align*}
  In other words, $-m+re^{\iu \alpha} = \xrep+\iu \yimp \in \mars_m^+$.
\end{proof}

\begin{thm}[MARS is convex]
  \label{thm:mars-is-convex}
  For every $m \in \bbN$, the sets $\mars_m$ are convex. 
  The increasing union $\mars=\bigcup_{m \in \bbN} \mars_m$
  is therefore also convex.
\end{thm}
\begin{proof}
    It is enough to show that the scaled and translated sets
    $\frac{1}{m}(m+\mars_m)$ are convex.
    By \Cref{lem:upper-boundary-g} we have
    to show that the curve bounding
    $\frac{1}{m}(m+\mars^+_m)$ from above
    has positive signed curvature when 
    traversed from $\frac{m+1}{m}$ to $0$.
    By \Cref{prop:mars-in-polar} this
    curve is given in polar coordinates
    by $\alpha \mapsto r_m(\alpha)e^{\iu\alpha},
    \alpha \in (0,\frac{\pi}{m+1})$
    with $r_m(\alpha) = \frac{\sin((m+1)\alpha)}{\sin(m\alpha)}$.
    If we set $\ufunc_m:=\frac{1}{r_m}$, then this
    curve has positive signed curvature if and only if $\ufunc_m+\ufunc_m''>0$,
    see e.g. \cite[p. 251]{mcmahonElementsDifferentialCalculus1898}.
    This inequality is established in \Cref{lem:curvature-polar}.
\end{proof}

\begin{rem}
    \label{rem:geometric-interpretation}
    There is a nice elementary geometric interpretation
    of the characterization of $\mars_m^+, m \in \bbN$ in \Cref{prop:mars-in-polar}.
    Namely, define points $A:=-m$, $B:=1$
    and $D:=0$ in $\bbC$.
    Then the curve bounding $\mars_m^+$ from above
    consists precisely of those points $C$ for which
    $0<\angle  CDB = (m+1) \angle CAB < \pi$. This
    follows from a simple application
    of the law of sines to the triangle ADC, see
    \Cref{fig:sine-triangle}.
\end{rem}

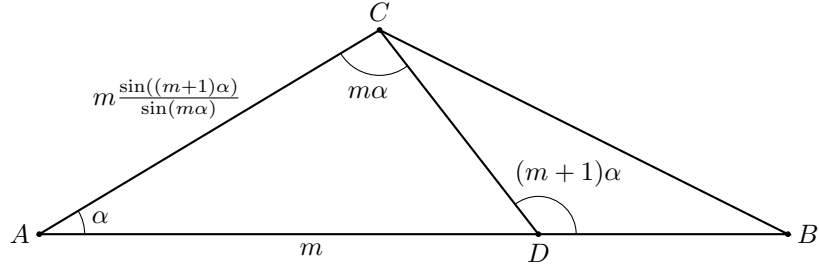
\begin{figure}
  \begin{center}
\begin{tikzpicture}[scale=0.9]

\coordinate (C) at (1,3);
\coordinate (A) at (-4,0);
\coordinate (B) at (7,0);

\def\m{2}
\coordinate (D) at ($(B)!{1/(1+\m)}!(A)$);

\draw[thick] (A)--(B)--(C)--cycle;
\draw[thick] (C)--(D);

\fill (A) circle (1.2pt);
\fill (B) circle (1.2pt);
\fill (C) circle (1.2pt);
\fill (D) circle (1.2pt);

\node[left] at (A) {$A$};
\node[right] at (B) {$B$};
\node[above] at (C) {$C$};
\node[below] at (D) {$D$};

\node[below] at (0,0) {$m$};
\node[left,xshift=-0.2cm] at ($(C)!{1/(1+\m)}!(A)$) {$m\frac{\sin((m+1)\alpha)}{\sin(m\alpha)}$};
\pic[
    draw,
    "$\alpha$",
    angle radius=6mm,
    angle eccentricity=1.4
] {angle=B--A--C};

\pic[
    draw,
    "$(m+1)\alpha$",
    angle radius=5mm,
    angle eccentricity=1.8
] {angle=B--D--C};

\pic[
    draw,
    "$m\alpha$",
    angle radius=6mm,
    angle eccentricity=1.4
] {angle=A--C--D};

\end{tikzpicture}
\end{center}
\caption{The triangle from \Cref{rem:geometric-interpretation}.}
\label{fig:sine-triangle}
\end{figure}

\section{Moving-Average Iteration in Banach Spaces}
\label{sec:Banach-spaces}

We conclude our paper by extending \Cref{thm:mas-via-mars} to
an infinite-dimensional setting.
Let $A$ be a bounded linear operator on a Banach space $X$.
Then $\calM_m(A)$ defined as in 
\eqref{eq:def-ma}
is a bounded linear operator on
$X^m$. The main result of this section
shows that containment of the spectrum of $A$ in $\mars_m$ characterizes uniform exponential stability of $\calM_m(A)$, which is equivalent to
$\lim_{n \to \infty} \norm{\calM_m(A)^n} = 0$ 
as shown for example in \cite[Chapter II, Proposition 1.3]{eisnerStabilityOperatorsOperator2010}

\begin{thm}[The MARS of operators in Banach space]
  \label{thm:mars-banach-space}
  Let $m \in \bbN$. Then $\lim_{n \to \infty} \norm{\calM_m(A)^n} =0$
  if and only if the spectrum $\sigma(A)$ is contained in $\mars_m$.
\end{thm}

The key result connecting the spectra of $A$ and $\calM_m(A)$
is the following

\begin{lem}
  \label{lem:connection-spectra}
  \begin{enumerate}[(a)]
  \item If $(v^{n})_{n \in \bbN}$ is an approximate eigenvector of $A$
    for the approximate eigenvalue $a$ and $\lambda \in \sigma(\calM_m(a))$, then
    the sequence $(\tilde{v}^{n})_{n \in \bbN}$ defined by
    $\tilde{v}^{n}:=(1,\lambda,\dots,\lambda^{m-1}) \otimes v^{n}$ is an
    approximate eigenvector of $\calM_m(A)$ for the approximate
    eigenvalue $\lambda$.
  \item If $(\tilde{v}^n)_{n \in \bbN}$ with
      $\tilde{v}^{n}=(\tilde{v}_1^{n},\dots,\tilde{v}_m^{n}) \in X^m$
    is an approximate eigenvector of $\calM_m(A)$ for the approximate
    eigenvalue $\lambda$, then $(\tilde{v}_1^{n})_{n \in \bbN}$ is an
    approximate eigenvector of $A$ for an approximate eigenvalue $a$
    with $\lambda \in \sigma(\calM_m(a))$.
  \end{enumerate}
\end{lem}
\begin{proof}
  \begin{enumerate}[(a)]
  \item Let $(v^{n})_{n \in \bbN}$ be an approximate eigenvector of
    $A$ for the approximate eigenvalue $a$.
    Let $\lambda \in \sigma(\calM_m(a))$, in other words
    $p_{m,a}(\lambda)=0$ and therefore
    $\lambda^{m} = \frac{a}{m} \sum_{\ell=0}^{m-1}
    \lambda^\ell$.
    There is a sequence of vectors $w^{n} \in X$ converging to
    $0$ such that
    $Av^{n}=av^n+w^{n}$.
    If we define $\tilde{v}^n := (1,\lambda,\dots,\lambda^{m-1}) \otimes v^{n}$, $n\in\bbN$, then
    \begin{align*}
      \calM_m(A) \tilde{v}^{n} &= \mat{\lambda v^n \\ \vdots \\
      \lambda^{m-1}v^n \\ \sum_{\ell=0}^{m-1}
      \lambda^\ell(\frac{a}{m}v^n+w^n) } 
      =\lambda \tilde{v}^n + \mat{0\\\vdots\\0\\ \sum_{\ell=0}^{m-1}\lambda^\ell w^n} \to
      \lambda \tilde{v}^n.
    \end{align*}
    Therefore $(\tilde{v}^{n})_{n \in \bbN}$ is an approximate
    eigenvector of
    $\calM_m(A)$ for the approximate eigenvalue $\lambda$.
  \item Let $(\tilde{v}^{n})_{n \in \bbN}$ be an approximate eigenvector of
    $\calM_m(A)$ for the approximate eigenvalue $\lambda$. There are
    vectors $w^{n}\in X^m$ such that
    $\calM_m(A) \tilde{v}^n = \lambda \tilde{v}^n+w^{n}$.  The structure of $\calM_m(A)$ implies
    \begin{align*}
      \tilde{v}_2^{n} &= \lambda \tilde{v}_1^{n} +w_1^{n}\\
      \tilde{v}_3^{n} &= \lambda^2 \tilde{v}_1^{n}+\lambda w_1^{n}+w_2^{n} \\
      &\vdots \\
      \tilde{v}_m^{n} &= \lambda^{m-1} \tilde{v}_1^{n} + \sum_{\ell=1}^{m-1}
                  \lambda^{m-1-\ell} w_\ell^{n}\\
      \sum_{\ell=0}^{m-1} A \frac{\lambda^\ell}{m} \tilde{v}_1^{n} +
      \tilde{w}^{n} &= \lambda \tilde{v}_m^{n} + w_m^{n} 
    \end{align*}
    with $\tilde{w}^{n} \to 0$ for $n \to \infty$.
    Therefore
    \begin{align*}
      \sum_{\ell=0}^{m-1} \frac{\lambda^\ell}{m} A \tilde{v}_1^{n} -\lambda^m
      \tilde{v}_1^{n} \to 0.
    \end{align*}
    Hence $0 \in \sigma_{\text{approx}}(\sum_{\ell=0}^{m-1} A \frac{\lambda^\ell}{m} -
    \lambda^m)$, where $\sigma_{\text{approx}}$ denotes the
    approximate point spectrum,
    and $\lambda^m \in (\sum_{\ell=0}^{m-1}
    \frac{\lambda^\ell}{m}) \sigma_{\text{approx}}(A)$.
    Thus there is $a \in \sigma_{\text{approx}}(A)$
    with $\lambda^m = (\sum_{\ell=0}^{m-1}
    \frac{\lambda^\ell}{m})a$. This is equivalent to $\lambda \in \sigma(\calM_m(a))$. \qedhere
  \end{enumerate}
\end{proof}

\begin{proof}[Proof of {\Cref{thm:mars-banach-space}}]
  We first show that $\sigma(A) \subseteq \mars_m$ if and only if
  $\partial \sigma(A) \subseteq \mars_m$.
  Since $\sigma(A)$ is closed, $\partial \sigma(A) \subseteq
  \sigma(A)$,
  so one implication is clear.
  For the other implication assume that
  $\partial \sigma(A) \subseteq \mars_m$
  and that $\xrep{}+\iu \yimp{} \in \sigma(A)$.
  There are $\yimp^-, \yimp^+$ with $\yimp^-\leq \yimp{} \leq \yimp^+$
  and $\xrep{}+\iu\yimp{}^-, \xrep{}+\iu\yimp{}^+ \in \partial \sigma(A)$. 
  By \Cref{thm:monotonicity-mars} and \eqref{eq:Fset-char} the
  intersection of $\mars_m$ with any line parallel to the imaginary
  axis is an interval. Hence $\xrep{}+\iu \yimp{} \in \mars_m$.

  The topological boundary of the spectrum is contained
  in the approximate point spectrum, see e.g.
  \cite[Chapter VII, Proposition 6.7]{conwayCourseFunctionalAnalysis1990}.
  By \cite[Chapter II, Proposition 1.3]{eisnerStabilityOperatorsOperator2010} $\calM_m(A)$
  is uniformly exponentially stable if and only if
  $\rho(\calM_m(A))<1$.

  Assume that $\sigma(A)\not\subseteq \mars_m$.
  By the argument above there is an approximate eigenvalue
  $a \not \in \mars_m$ of $A$. By the definition
  of $\mars_m$ there must be $\lambda \in \sigma(\calM_m(a))$
  with $\abs{\lambda}\geq 1$.
  By \Cref{lem:connection-spectra} (a)
  this implies $\rho(\calM_m(A))\geq 1$
  and therefore $\calM_m(A)$ is not uniformly
  exponentially stable.

  On the other hand, assume that  
  $\rho(\calM_m(A))\geq 1$. There must be
  an approximate eigenvalue $\lambda$
  of $\calM_m(A)$ with absolute value at least $1$.
  By \Cref{lem:connection-spectra} (b)
  there is $a \in \sigma(A)$ with $\lambda \in \sigma(\calM_m(a))$.
  Hence $\sigma(A) \not \subseteq \mars_m$. 
\end{proof}

\section{Conclusion}
\label{sec:conclusion}

In this paper we have introduced the stability regions for linear moving-average iteration and have derived explicit formulas describing these sets. 
We have shown in
  \Cref{prop:simple-properties}
  that the stability of moving-average iteration
of a matrix $A$ is controlled by
the inclusion of the spectrum $\sigma(A)$ in an open
region of the complex plane. 
The same was shown for
Krasnoselskii iteration under reasonable assumptions
on the weight sequence in \Cref{thm:stability-Krasnoselskii}.
However, as \Cref{exam:non-open-stability-region} shows, some
assumptions on the weight sequence 
are indispensable
to obtain a characterization of stability of the iteration scheme
via an open stability region. This naturally raises the following
two questions.
\begin{quest}
  \label{qu:open-stability-set}
  For which Mann-weight sequences is there an open region
  $\calS \subseteq \bbC$ such that the Mann iteration
  of a square matrix $A$ is stable if and only if $\sigma(A) \subseteq
  \calS$?
\end{quest}

\begin{quest}
  \label{qu:regions}
  Which subsets of the complex plane can arise as stability
  sets of Mann iterations? Are these stability regions all convex?
\end{quest}

The classical Mann iteration,
Krasnoselskii iteration and
moving-average iteration
all correspond to
Mann-weight sequences $(\pi^k)_{k \in \bbN_0}$ of the form
\begin{align*}
  \pi^k_\ell = \frac{1}{\sum_{j=0}^k \zeta_j} \zeta_{k-\ell}
\end{align*}
% x_{k+1} = A \sum_{\ell=0}^k \xi_{k-\ell} x_\ell  
for a sequence $(\zeta_j)_{j \in \bbN_0}$ of non-negative real numbers
with $\zeta_0\neq 0$.
This class of Mann iterations might be the natural
class for which one should approach \Cref{qu:open-stability-set}
and \ref{qu:regions}.

Additionally, it would be interesting to derive conditions on a matrix
that imply the containment of its spectrum in $\mars_m$
or $\mars$ e.g. in the spirit of the
Schur-Cohn theorem.
A good starting point for this question might be
\cite{lasserreCharacterizingPolynomialsRoots2004}.

Finally, iteration schemes as discussed in this paper become
very relevant when they are applied to noisy dynamics.
Krasnoselskii iteration is well studied in this setting under the name
\enquote{stochastic approximation}, see e.g.
\cite{borkarStochasticApproximationDynamical2023}
and \cite{kushnerStochasticApproximationRecursive2003}, but
the behavior of moving-average iteration,
in particular in the limit of large window sizes,
seems to be much less explored.

\printbibliography

\appendix

\section{Various Estimates}
\label{sec:various-estimates}

The following lemma can be seen as a generalization of
the inscribed angle theorem, which corresponds to $m=1$,
in which case the inequality in the assertion actually
becomes an equality.
\begin{lem}
  \label{lem:right-angle-split-m}
  Let $m\geq 1$.
  Let $ABC$ be a right-angled triangle with the right angle at $C$.
  Let $D$ be a point on the side $AB$ with $m \abs{BD}=\abs{DA}$.
  Let $\alpha$ be the interior angle of $ABC$ at $A$
  and $\delta$ the interior angle of $CDB$ at $D$.
  Then $\delta\leq (m+1) \alpha$.
\end{lem}
\begin{proof}
  Let $E$ be the point on the segment $\overline{BC}$ such that
  $DE$ is perpendicular to $BC$, see \Cref{fig:right-triangle}.
  Since $ED$ is parallel to $AC$ we
  have $\angle  BDE = \angle BAC = \alpha$
  and $\angle EDC = \delta-\alpha$.
  By the intercept theorem
  we have $\frac{\abs{EC}}{\abs{BE}} = \frac{\abs{DA}}{\abs{BD}}=m$.
  Furthermore
  \begin{align*}
    \tan(\delta-\alpha) &= \tan{\angle EDC} = \frac{\abs{EC}}{\abs{ED}}, \\
    \tan(\alpha) &=\tan{\angle BDE} = \frac{\abs{BE}}{\abs{ED}}, \\
    \frac{\tan(\delta-\alpha)}{\tan(\alpha)} &=
                                               \frac{\abs{EC}}{\abs{BE}}=m
                                               \geq 1.
  \end{align*}
  Now $\alpha$ and $\delta-\alpha$
  both lie in $(0,\frac{\pi}{2})$ since they appear
  as interior angles of right-angled triangles.
  Since $\tan$ is monotonically increasing on
  $(0,\frac{\pi}{2})$, we have $\delta-\alpha \geq \alpha$.
  Since $\tan$ is also convex on $(0,\frac{\pi}{2})$
  we have
  \begin{align*}
    \frac{\delta-\alpha}{\alpha} \leq \frac{\tan(\delta-\alpha)}{\tan(\alpha)}=m,
  \end{align*}
  which immediately implies the assertion.  
\end{proof}
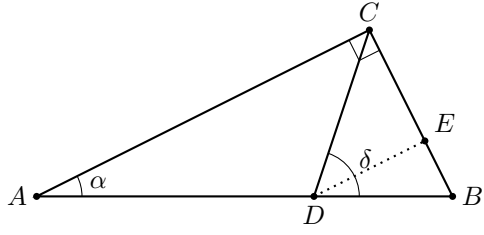
\begin{figure}
  \begin{center}
\begin{tikzpicture}[scale=1.1]
\usetikzlibrary{angles,quotes,calc}

\coordinate (C) at (4,2);
\coordinate (A) at (0,0);
\coordinate (B) at (5,0);

\def\m{2}
\coordinate (D) at ($(B)!{1/(1+\m)}!(A)$);
\coordinate (E) at ($(B)!{1/(1+\m)}!(C)$);

\draw[thick] (A)--(B)--(C)--cycle;
\draw[thick] (C)--(D);
\draw[dotted,thick] (D)--(E);
\pic[draw,
    angle radius=3mm] {right angle=A--C--B};

\fill (A) circle (1.2pt);
\fill (B) circle (1.2pt);
\fill (C) circle (1.2pt);
\fill (D) circle (1.2pt);
\fill (E) circle (1.2pt);

\node[left] at (A) {$A$};
\node[right] at (B) {$B$};
\node[above] at (C) {$C$};
\node[below] at (D) {$D$};
\node[above right] at (E) {$E$};

\pic[
    draw,
    "$\alpha$",
    angle radius=6mm,
    angle eccentricity=1.4
] {angle=B--A--C};

\pic[
    draw,
    "$\delta$",
    angle radius=6mm,
    angle eccentricity=1.4
] {angle=B--D--C};
\end{tikzpicture}
\end{center}
\caption{The triangle from \Cref{lem:right-angle-split-m}.}
\label{fig:right-triangle}
\end{figure}

The next two lemmas prove the facts about $r\cot(r)$
used in \Cref{lem:y-coty}.
\begin{lem}
  \label{lem:ycoty-bijective}
  The function
\begin{align*}
    \varphi: (0,\pi) \to (-\infty,1), \;\yvar{} \mapsto \yvar{}\cot(\yvar{})
\end{align*}
is bijective.
\end{lem}
\begin{proof}
    Since
    \begin{align}
        \cot'(\yvar{})=-\frac{1}{\sin^2(\yvar{})}=-(1+\cot^2(\yvar{}))
    \end{align}
    we have
    \begin{align*}
      \varphi'(\yvar{})=\cot(\yvar{})-\frac{\yvar{}}{\sin^2(\yvar{})}.
    \end{align*}

    From $2\cos(\yvar{})\sin(\yvar{})=\sin(2\yvar{})<2\yvar{}$ for $\yvar{}>0$
    we obtain 
    \begin{align*}
      \cot(\yvar{})<\frac{\yvar{}}{\sin^2(\yvar{})}
    \end{align*}
    and hence $\varphi'(\yvar{})<0$ for $\yvar{} \in (0,\pi)$.
    This implies injectivity of $\varphi$.
    Surjectivity
    follows from the continuity of $\varphi$ together with $\lim_{\yvar{}
      \to \pi} \varphi(\yvar{})=-\infty$
    and 
    \begin{align*}
      \lim_{\yvar{} \to 0} \yvar{}\cot(\yvar{}) = \lim_{\yvar{} \to 0} \frac{\cos(\yvar{})-\yvar{}\sin(\yvar{})}{\cos(\yvar{})}=1.
    \end{align*}
  \end{proof}

\begin{lem}
  \label{lem:lower-bound-fm}
  For every $m\geq 1$ and $\yvar{} \in (0,\pi)$ we have
  \begin{align}
  \label{eq:lower-bound-fm}
    \xvar{} \cot\left(\frac{\xvar{}}{m}\right) - \xvar{} \cot(\xvar{}) \geq m-1.
  \end{align}
\end{lem}
\begin{proof}
   Consider the map $\psi: (0,\pi) \to \bbR$
   defined by $\psi(\xvar):= \cot(\xvar) - \frac{1}{\xvar}$.
   Let $m\geq 1$, $\xvar \in (0,\pi)$.
   The well-known estimate
   $\sin(\xvar)<\xvar$ 
   implies $\psi'(\xvar)=-\frac{1}{\sin^{2}(\xvar)}+\frac{1}{\xvar^2}<0$.
   Therefore $\psi$ is decreasing
   and we get $0\leq \psi(\frac{r}{m})-\psi(r)=\cot(\frac{r}{m})-\frac{m}{r}-\cot(r)+\frac{1}{r}$.
   Multiplying this inequality by $r$ and rearranging gives \eqref{eq:lower-bound-fm}.
\end{proof}
Next we prove the quadratic upper bound for
$\frac{4}{\xvar{}+3}$ used in the proof of \Cref{thm:rational-func-description}.
\begin{lem}
  \label{lem:R2-bound-1}
  $\frac{4}{\xvar{}+3}<\xvar{}^2+3(\xvar{}+2)^2$ for $\xvar{} \in (-2,1)$.
\end{lem}
\begin{proof}
  Set $f(\xvar{}) = \xvar{}^2+3(\xvar{}+2)^2-\frac{4}{\xvar{}+3}, \xvar{} \in [-2,1]$.  Then
  \begin{align*}
    f'(\xvar{})&=8\xvar{}+12+\frac{4}{(\xvar{}+3)^2}\\
    f''(\xvar{})&=8-\frac{8}{(\xvar{}+3)^3}
  \end{align*}
  $f''$ is strictly increasing on $(-2,1)$
  and $f''(-2)=0$.
  Therefore $f'$ is strictly increasing in $(-2,1)$
  and $f'(-2)=0$.
  This in turn shows that $f$ is strictly
  increasing on $(-2,1)$
  and together with $f(-2)=0$ this proves our claim.
\end{proof}

We finish with two trigonometric calculations used
in \Cref{sec:convexity}.
\begin{lem}
\label{lem:cotm-cotm}
For $m \in \bbN$ and $\alpha \in (0,\frac{\pi}{m+1})$
we have
\begin{align*}
    \cot(\alpha)-\cot((m+1)\alpha) &=\frac{\sin(m\alpha)}{\sin(\alpha)\sin((m+1)\alpha)}.
  \end{align*} 
\end{lem}
\begin{proof}
The trigonometric identity
  \begin{align*}
    \cot((m+1)\alpha) = \frac{\cot(m\alpha)\cot(\alpha)-1}{\cot(m\alpha)+\cot(\alpha)}
  \end{align*}
  implies
  \begin{align*}
    \cot(\alpha)-\cot((m+1)\alpha) &= \frac{\cot^2(\alpha)+1}{\cot(m\alpha)+\cot
    \alpha}=  \frac{1}{\sin^2(\alpha)(\cot(m\alpha)+\cot
                                     \alpha)} \\
    &=\frac{\sin(m\alpha)}{ \sin(\alpha)\left(\sin(\alpha)\cos(m\alpha) + \cos(\alpha)
      \sin(m\alpha)\right)} \\
                                   &=\frac{\sin(m\alpha)}{\sin(\alpha)\sin((m+1)\alpha)}.
                                   \qedhere
  \end{align*} 
\end{proof}

\begin{lem}
\label{lem:curvature-polar}
  For $m \in \bbN$ define  $\ufunc_m:
  (0,\frac{\pi}{m+1})\to \bbR, \ufunc_m(\alpha) =
  \frac{\sin(m\alpha)}{\sin((m+1)\alpha)}$.
  Then $\ufunc_m''+\ufunc_m>0$.
\end{lem}
\begin{proof}
  Let $m \in \bbN$ and set $\tilde{m}:=m+1$ to simplify notation.
  By \Cref{lem:cotm-cotm} we have
  \begin{align*}
    \ufunc_m(\alpha) &= \cos(\alpha)-\sin(\alpha) \cot(\tilde{m}\alpha), \\
    \ufunc'_m(\alpha) &= -\sin(\alpha)
    -\cos(\alpha)\cot(\tilde{m}\alpha)
    +\tilde{m}\sin(\alpha)(1+\cot^2(\tilde{m}\alpha)),
    \\
    \ufunc''_m(\alpha) &= -\cos(\alpha)+\sin(\alpha)\cot(\tilde{m}\alpha)
                         +\tilde{m}\cos(\alpha)(1+\cot^2(\tilde{m}\alpha))\\
                     &\phantom{=}
                       +\tilde{m}\cos(\alpha)(1+\cot^2(\tilde{m}\alpha))\\
                     &\phantom{=}
                       -2\tilde{m}^2\sin(\alpha)\cot(\tilde{m}\alpha)(1+\cot^2\tilde({m}\alpha)),\\
    \ufunc_m(\alpha)+\ufunc''_m(\alpha) &=
                       2\tilde{m}\frac{\sin(\alpha)}{\sin^2(\tilde{m}\alpha)}
                       (\cot(\alpha)-\tilde{m}\cot(\tilde{m}\alpha)).
  \end{align*}
  Now $\tan$ is convex on $(0,\frac{\pi}{2})$,
  hence for $\alpha \in (0,\frac{\pi}{2\tilde{m}})$ we have
  \begin{align*}
    \frac{1}{\tilde{m}}\tan(\tilde{m}\alpha) > \tan(\alpha), \\
    \tilde{m}\cot(\tilde{m}\alpha) < \cot(\alpha).
  \end{align*}
  For $\alpha \in [\frac{\pi}{2\tilde{m}},\frac{\pi}{\tilde{m}})$
  on the other hand, $\tilde{m}\cot(\tilde{m}\alpha) < \cot(\alpha)$
  holds trivially since the right hand side is positive while the left hand side is nonpositive.
  This shows $u_m(\alpha)+u''_m(\alpha)>0$.
\end{proof}

\end{document}